\documentclass[twoside]{article}

\usepackage[accepted]{aistats2022}
\usepackage{natbib}
\usepackage{graphicx}

\usepackage[english]{babel}
\usepackage[utf8]{inputenc}
\usepackage[T1]{fontenc}
\usepackage{ragged2e} 
\usepackage{stackengine}
\usepackage[margin=1in]{geometry}

\usepackage{amsmath,amssymb,amsthm,bm}
\usepackage{latexsym,color,minipage-marginpar,caption,multirow,verbatim}
\usepackage{times}
\usepackage{bbm}
\usepackage{algorithm}
\usepackage{algpseudocode}
\usepackage{amsmath, array, tikz}
\usepackage{enumitem}

\definecolor{darkblue}{rgb}{0.2, 0.2, 0.5}

\newtheorem{theorem}{Theorem}
\newtheorem{corollary}{Corollary}
\newtheorem{definition}{Definition}
\newtheorem{proposition}{Proposition}
\newtheorem{assumption}{Assumption}
\newtheorem{lemma}{Lemma}
\newtheorem{example}{Example}
\newtheorem*{remark}{Remark}

\usepackage{xcolor}

\definecolor{tumb}{RGB}{0,101,189}

\begin{document}

%
\runningtitle{Exponential Number of Spurious Solutions and Failure of Gradient Methods}

%
\runningauthor{Yalcin, Zhang, Lavei, Sojoudi}

\twocolumn[

\aistatstitle{Factorization Approach for Low-complexity Matrix Completion Problems: Exponential Number of Spurious Solutions and Failure of Gradient Methods}

\aistatsauthor{ Baturalp Yalcin \And Haixiang Zhang \And Javad Lavaei \And Somayeh Sojoudi }

\aistatsaddress{ UC Berkeley \And  UC Berkeley \And UC Berkeley \And UC Berkeley }]

\begin{abstract}
    It is well-known that the Burer-Monteiro (B-M) factorization approach can efficiently solve low-rank matrix optimization problems under the RIP condition. It is natural to ask whether B-M factorization-based methods can succeed on any low-rank matrix optimization problems with a low information-theoretic complexity, i.e., polynomial-time solvable problems that have a unique solution. In this work, we provide a negative answer to the above question. We investigate the landscape of B-M factorized polynomial-time solvable matrix completion (MC) problems, which are the most popular subclass of low-rank matrix optimization problems without the RIP condition. We construct an instance of polynomial-time solvable MC problems with exponentially many spurious local minima, which leads to the failure of most gradient-based methods. Based on those results, we define a new complexity metric that potentially measures the solvability of low-rank matrix optimization problems based on the B-M factorization approach. In addition, we show that more measurements of the ground truth matrix can deteriorate the landscape, which further reveals the unfavorable behavior of the B-M factorization on general low-rank matrix optimization problems.
\end{abstract}

\section{INTRODUCTION}

The low-rank matrix optimization problem aims to recover a low-rank ground truth matrix $\mathbf{M^*}$ through some measurements modeled as $\mathcal{A}({\bf M}^*)$, where the measurement operator $\mathcal{A}$ is a function from $\mathbb{R}^{n \times n}$ to $\mathbb{R}^d$.
The operator $\mathcal{A}$ can be either linear as in the linear matrix sensing problem and the matrix completion problem \citep{candes2009exact,recht2010guaranteed}, or nonlinear as in the one-bit matrix sensing problem \citep{davenport20141} and the phase retrieval problem \citep{shechtman2015phase}. There are two variants of the problem, known as symmetric and asymmetric  problems. The first one assumes that $\mathbf{M}^*$ is a positive semi-definite (PSD) matrix, whereas the second one makes no such assumption and allows $\mathbf{M}^*$ to be non-symmetric or sign indefinite. Since the asymmetric problem can be equivalently transformed into a symmetric problem \citep{zhang2021general}, we focus on the latter one.

There are in general two different approaches to overcome the non-convex low-rank constraint. The first approach is to design a convex penalty function that prefers low-rank matrices and then optimize the penalty function under the measurement constraint \citep{candes2009exact,recht2010guaranteed,candes2010power}. However, this approach works in the matrix space $\mathbb{R}^{n\times n}$ and has a high computational complexity. The other widely accepted technique is the Burer-Monteiro (B-M) factorization approach \citep{burer2003nonlinear}, which converts the original problem into an unconstrained one by replacing the original PSD matrix variable $\mathbf{M} \in \mathbb{R}^{n \times n}$ with the product of a low-dimensional variable $\mathbf{X} \in \mathbb{R}^{n \times r}$ and its transpose. The optimization problem based on the B-M factorization approach can be written as
\begin{align*}
    \min_{{\bf X}\in\mathbb{R}^{n\times r}}~ g\left[ \mathcal{A}({\bf XX}^T) - \mathcal{A}({\bf M}^*) \right],
\end{align*}
where $g(\cdot)$ is a loss function that penalizes the mismatch between $\mathbf{XX}^T$ and ${\bf M}^*$.
Using the B-M factorization, the objective function is generally non-convex even if the loss function $g(\cdot)$ is convex. Nonetheless, it has been proved that under certain strong conditions, such as the Restricted Isometry Property (RIP) condition, saddle-escaping methods can converge to the ground truth solution with a random initialization \citep{zhang2021general,bi2021local} and first-order methods with spectral initialization converge locally \citep{tu2016low,bhojanapalli2016dropping}; see \cite{chi2019nonconvex} for an overview. 

Then, it is natural to ask whether optimization methods based on the B-M factorization approach can succeed on general low-rank matrix optimization problems with a low information-theoretic complexity (i.e., problems that have a unique global solution and can be solved in polynomial time), especially when the RIP condition does not hold. In this work, we focus on a common class of problems without the RIP condition, namely the matrix completion (MC) problem. For the MC problem, the measurement operator $\mathcal{A}_\Omega:\mathbb{R}^{n\times n} \mapsto \mathbb{R}^{n\times n}$ is given by
\[ \mathcal{A}_\Omega({\bf M})_{ij} := \begin{cases} {\bf M}_{ij} & \text{if } (i,j) \in \Omega\\ 0 & \text{otherwise}, \end{cases} \]
where $\Omega$ is the set of indices of observed entries. We denote the measurement operator as ${\bf M}_\Omega := \mathcal{A}_{\Omega}({\bf M})$ for simplicity. An instance of the MC problem, denoted as $\mathcal{P}_{\mathbf{M}^*, \Omega, n, r}$, can be formulated as
\begin{align*}
    \mathrm{find}&\quad {\bf M}\in\mathbb{R}^{n\times n} \tag{$\mathcal{P}_{\mathbf{M}^*, \Omega, n, r}$} \\
    \mathrm{s.t.}&\quad \mathrm{rank}({\bf M}) \leq r,\quad{\bf M} \succeq 0, \quad {\bf M}_\Omega = {\bf M}^*_\Omega.
\end{align*}
%
%
If ${\bf M}^*$ is the only solution of this problem, we will say that $\mathcal{P}_{\mathbf{M}^*, \Omega, n, r}$ has a unique solution. Using the B-M factorization approach, the MC problem can be solved via the optimization problem
\begin{align}\label{eqn:obj-general}
    \min_{{\bf X} \in \mathbb{R}^{n\times r}} f({\bf X}),
\end{align}
where $f({\bf X}) := g[(\mathbf{XX}^T  - {\bf M}^*)_\Omega]$.
%
%
For example, if the $\ell_2$-loss function is used, the problem becomes
\begin{align}\label{eqn:obj-linear}
    \min_{{\bf X} \in \mathbb{R}^{n\times r}} \left\|(\mathbf{XX}^T - {\bf M}^*)_\Omega\right\|_F^2.
\end{align}
%
%
\paragraph{Contributions.} 
We provide a negative answer to the preceding question by constructing MC problem instances for which the \textit{optimization complexity} of local search methods using the B-M factorization does not align with the \textit{information-theoretic complexity} of the underlying MC problem instance. The information-theoretic complexity refers to the minimum number of operations that the best possible algorithm takes to find the ground truth matrix, while the optimization complexity refers to the minimum number of operations that a given optimization method takes to find the ground truth matrix. In general, the optimization complexity of local search methods depends on the properties of spurious solutions of the optimization problem, e.g., the number, the sharpness and the regions of attraction of spurious solutions. The optimization complexity predicts the performance of an algorithm and provides a hint on which algorithm to use for a given problem. Therefore, the results in this work imply that the popular B-M factorization approach is not able to capture the benign properties of the low-rank problem when the RIP condition does not hold. We summarize our contributions as follows:
\begin{enumerate}[label=\bf \roman*),noitemsep]
    \item Given natural numbers $n$ and $r$ with $n\geq 2r$, we construct a class of MC problem instances $\mathcal{L}(\mathcal{G}, n, r)$, whose ground matrix ${\bf M}^*\in\mathbb{S}_+^n$ has rank $r$. For every instance in this class, there exists a unique global solution and the solution can be found in polynomial time via graph-theoretic algorithms.
    \item Next, we show the existence of an instance in $\mathcal{L}(\mathcal{G}, n, r)$ whose B-M factorization formulation \eqref{eqn:obj-linear} has at least $\mathcal{O}(2^{n-2r})$ equivalent classes of spurious\footnote{A solution is called spurious if it is a local minimum but has a larger objective value than the optimal objective value.} local solutions. Note that this claim holds for general loss functions under a weak assumption.
    \item Moreover, for the rank-$1$ case, we prove that most gradient-based methods with a random initialization converge to a spurious local minimum with probability at least $1-\mathcal{O}(2^{-n/2})$. Numerical studies verify that the failure of the gradient-based methods also happens for general rank cases.
    \item We present an instance that has no spurious solution under the B-M factorization formulation \eqref{eqn:obj-linear}, but introducing additional observations of the ground truth matrix leads to at least exponentially many spurious solutions. This example further reveals the unfavorable behavior of the B-M factorization approach on general low-rank matrix optimization problems.
\end{enumerate}
%
Based on these results, we define a new complexity metric that potentially captures the optimization complexity of optimization methods based on the B-M factorization.

\paragraph{Related Work.} The low-rank optimization problem has been well studied under the RIP condition \citep{recht2010guaranteed}. Several recent works \citep{zhang2019sharp,bi2020global,zhang2021general} showed that the non-convex formulation has no spurious local minima with a small RIP parameter. To understand how conservative the RIP condition is, we consider a class of polynomial-time solvable problems without the RIP condition and study the behavior of optimization methods on this class. Specifically, we consider the polynomial-time solvable MC problems.
Most existing literature on the MC problem is based on the assumption that the measurement set is randomly constructed and the global solution is coherent \citep{candes2009exact,candes2010power,ge2016matrix,ge2017no,ma2019implicit,chen2020nonconvex}. In comparison, there is a small range of works that have focused on the deterministic MC problem \citep{bhojanapalli2014universal,kiraly2015algebraic,pimentel2016characterization,li2016recovery}. Furthermore, efficient graph-theoretic algorithms utilizing the special structures of a deterministic measurement set can be designed \citep{ma2018gradient}. Existing works on a deterministic measurement set case have focused on the completability problem and the convex relaxation approach, while the B-M factorization approach has not been analyzed. 
Moreover, there are several existing works that also provided negative results on the low-rank matrix optimization problem. The counterexamples in \cite{candes2010power,bhojanapalli2014universal} have non-unique global solutions, which make the recovery of the ground truth matrix impossible. The counterexamples in \cite{waldspurger2020rank} have a unique global solution but the objective function must be a linear function. We refer the reader to \cite{chi2019nonconvex} for a review of the low-rank matrix optimization problem. Our work is the first one in the literature that studies the optimization complexity in the case when the information-theoretic complexity is low.

\paragraph{Notations.} 
The set $[n]$ represents the set of integers from $1$ to $n$. We use lower case bold letters $\mathbf{x}$ to represent vectors and capital bold letters $\mathbf{X}$ to represent matrices. $\|\mathbf{X}\|$ and $\|\mathbf{X}\|_F$ are the $2$-norm and the Frobenius norm of the matrix $\mathbf{X}$, respectively. Let $\langle {\bf A}, {\bf B}\rangle = \mathrm{Tr}({\bf A}^T{\bf B})$ be the inner product between matrices. The notations ${\bf X}\succeq 0$ and ${\bf X}\succ 0$ mean that the matrix ${\bf X}$ is PSD and positive definite, respectively. The set of $n \times n$ PSD matrices is denoted as $\mathbb{S}_+^n$.
For a function $f: \mathbb{R}^{m \times n} \mapsto \mathbb{R}$, we denote the gradient and the Hessian as $\nabla f(\cdot)$ and $\nabla^2f(\cdot)$, respectively. The Hessian is a four-dimensional tensor with $[\nabla^2 f(\mathbf{X})]_{i,j,k,l} = \frac{\partial^2 f(\mathbf{X})}{\partial \mathbf{X}_{ij} \partial \mathbf{X}_{k,l}}$ for all $i,j \in [m]$ and $k,l \in [n]$. The quadratic form of the Hessian in the direction $\Delta \in \mathbb{R}^{m \times n} $ is defined as $\Delta: \nabla^2 f(\mathbf{X}): \Delta = \sum_{i,j,k,l}[\nabla^2 f(\mathbf{X})]_{i,j,k,l}\Delta_{ij}\Delta_{kl}$. We use $\lceil\cdot\rceil$ and $\lfloor\cdot\rfloor$ to denote the ceiling and flooring functions, respectively. The cardinality of a set $\mathcal{S}$ is shown as $|\mathcal{S}|$.

\section{EXPONENTIAL NUMBER OF SPURIOUS LOCAL MINIMA}
\label{sec:exp-sol}

In this section, we show that MC problem instances with a low information-theoretic complexity may have exponentially many spurious local minima if the B-M factorization is employed. We first construct a class of MC problem instances with a low information-theoretic complexity and then identify the problematic instances.

\subsection{Low-complexity Class of MC Problems}

Suppose that $r\geq 1$ and $n\geq 2r$ are two given integers. We construct a class of MC problem instances whose ground truth matrix ${\bf M}^*\in\mathbb{S}^{n}_+$ is rank-$r$. For every instance in this class, the global solution is unique and can be found in polynomial time in terms of $n$ and $r$. Let $m := \lfloor n / r\rfloor \geq 2$. We divide the first $mr$ rows and the first $mr$ columns of the matrix ${\bf M}^*$ into $m\times m$ block matrices, where each block has dimension $r\times r$. For every $i,j\in[m]$, we denote the block matrix at position $(i,j)$ as ${\bf M}^*_{i,j}$. 
%
%
We now define the block measurement patterns induced by a given graph.
%
\begin{definition}[Induced measurement set]\label{blocksparsitygraph}
    Let $\mathcal{G} = (\mathcal{G}_1,\mathcal{G}_2)= (\mathcal{V}, \mathcal{E}_1, \mathcal{E}_2)$ be a pair of undirected graphs with the node set $\mathcal{V} = [m]$ and the disjoint edge sets $\mathcal{E}_1,\mathcal{E}_2  \subset [m]\times[m]$, respectively. The induced measurement set $\Omega(\mathcal{G})$ is defined as follows: if $(i,j) \in \mathcal{E}_1$, then the entire block ${\bf M}^*_{i,j}$ are observed; if $(i,j) \in \mathcal{E}_2$, then all nondiagonal entries of the block ${\bf M}^*_{i,j}$ are observed; otherwise, none of the entries of the block is observed. In addition, the last $n-mr$ rows and the last $n-mr$ columns of the matrix ${\bf M}^*$ are fully observed. We refer to the graph $\mathcal{G}$ as the block sparsity graph.
\end{definition}
The following definition introduces a low-complexity class of MC problem instances. 
\begin{definition}[Low-complexity class of MC problems]
    Define $\mathcal{L}(\mathcal{G}, n, r)$ to be the class of low-complexity MC problems $\mathcal{P}_{{\bf M}^*, \Omega, n, r}$ with the following properties:
    \setlist{nolistsep}
    \begin{enumerate}[label= \bf \roman*), noitemsep]
        \item The ground truth $\mathbf{M^*} \in \mathbb{S}^{n}_+$ is rank-r.
        \item The matrix $\mathbf{M}^*_{i,j}\in\mathbb{R}^{r\times r}$ is rank-$r$ for all $i,j\in[m]$.
        \item The measurement set $\Omega=\Omega(\mathcal{G})$ is induced by $\mathcal{G}=(\mathcal{G}_1,\mathcal{G}_2)$, where $\mathcal{G}_1$ is connected and non-bipartite.
    \end{enumerate}
\end{definition}
%
The next proposition states that every MC problem instance in $\mathcal{L}(\mathcal{G}, n, r)$ is polynomial-time solvable.
\begin{proposition}\label{prop:rankrpolynomial}
    For an arbitrary instance $\mathcal{P}_{\mathbf{M}^*, \Omega, n, r}$ in $\mathcal{L}(\mathcal{G}, n, r)$, the ground truth $\mathbf{M^*}$ is the unique solution of this problem and can be found in $\mathcal{O}(n^2/r^2 + nr^2)$ time.
\end{proposition}

\subsection{Intuition for Rank-$1$ Case with $\ell_2$-loss Function}
We start with the case when the rank $r$ is equal to $1$ and the loss function $g(\cdot)$ is the $\ell_2$-loss. We study two instances in the class $\mathcal{L}(\mathcal{G}, n, 1)$ with $\mathcal{O}(n)$ and $\mathcal{O}(n^2)$ observations, respectively. The B-M formulation \eqref{eqn:obj-linear} of both instances contains exponentially many spurious local minima. Since the decomposition variable ${\bf X}$ is a column vector in the rank-$1$ case, we write it as ${\bf x}$.

\begin{example}\label{exm:1}
We first provide an instance with $\mathcal{O}(n)$ observations. Note that the number of blocks, namely $m$, is equal to $n$ in the rank-$1$ case. Let the graph $\mathcal{G}=(\mathcal{V},\mathcal{E}_1, \mathcal{E}_2)$ be chosen as $\mathcal{V} := [n]$ and
\[ \mathcal{E}_1 := \{ (i,j) ~|~ i,j\in[n],~ |i-j|\leq 1 \},\quad \mathcal{E}_2:=\emptyset. \]
The measurement set is the induced set $\Omega := \Omega(\mathcal{G})$. Namely, we observe the diagonal, sub-diagonal and super-diagonal entries of the ground truth matrix.
One can verify that the subgraph $\mathcal{G}_1 = (\mathcal{V}, \mathcal{E}_1)$ is connected and non-bipartite. 
Now, we construct a specific ground truth matrix. We define the vector ${\bf x}^*\in\mathbb{R}^n$ by
\begin{align*} 
{\bf x}^*_{2k+1} &:= 1,~ \forall k\in \left[ \lceil n/2 \rceil \right],\quad {\bf x}^*_{2k} := 0,~ \forall k\in \left[ \lfloor n/2 \rfloor \right],
\end{align*}
%
and let ${\bf M}^*:={\bf x}^*({\bf x}^*)^T$. For the B-M factorization formulation \eqref{eqn:obj-linear}, the set of global minima is given by
\begin{align*} 
\mathcal{X}^* := \big\{ {\bf x}\in\mathbb{R}^n~|~  {\bf x}_{2k+1}^2 &= 1,\quad \forall k\in \left[ \lceil n/2 \rceil \right],\\
{\bf x}_{2k} &= 0,\quad \forall k\in \left[ \lfloor n/2 \rfloor \right] \big\},
\end{align*}
which has cardinality $2^{\lceil \frac{n}{2} \rceil}$. For every global solution $\hat{{\bf x}} \in \mathcal{X}^*$ and every $\Delta \in \mathbb{R}^n\backslash\{0\}$, the Hessian satisfies
\begin{align*} 
\Delta : \nabla^2 f(\hat{{\bf x}}) : \Delta &= 2\left\| \left( \hat{{\bf x}}\Delta^T + \Delta \hat{{\bf x}}^T \right)_\Omega \right\|_F^2\\
&= 8\|\Delta\|^2  - \mathbbm{1}[n \; \text{is even}]4\Delta_n^2> 0,
\end{align*}
where $\mathbbm{1}[\cdot]$ is the indicator function. Therefore, the Hessian is positive definite at every global minimum. Then, we perturb the ground truth solution ${\bf M}^*$ to
\[ {\bf M}^*(\epsilon) := {\bf x}^*(\epsilon) \left[{\bf x}^*(\epsilon)\right]^T = ({\bf x}^* + \epsilon)({\bf x}^* + \epsilon)^T, \]
where ${\bf x}^*(\epsilon) := {\bf x}^* + \epsilon$ and $\epsilon \in \mathbb{R}^n$ is a small perturbation.
We denote the associated problem \eqref{eqn:obj-linear} as
\begin{equation}\label{eqn:prob2}
\setlength{\abovedisplayskip}{3pt}\setlength{\belowdisplayskip}{3pt}
\min_{\mathbf{x}\in\mathbb{R}^n} \tilde{f}(\mathbf{x}; \epsilon),
\end{equation}
where $\tilde{f}(\mathbf{x}; \epsilon) := \| (\mathbf{xx}^T - \mathbf{M}(\epsilon)^*)_{\Omega} \|_F^2$. For a generic perturbation $\epsilon$, all components of $\epsilon$ are nonzero and problem $\mathcal{P}_{{\bf M}^*(\epsilon), \Omega,n,1}$ belongs to the class $\mathcal{L}(\mathcal{G}, n, 1)$. This implies that the global solution of problem \eqref{eqn:prob2} is unique up to a sign flip.

We analyze the relation between the local minima of the original problem and those of the perturbed problem. Consider the equation $\nabla_{\bf x} \tilde{f}({\bf x};\epsilon) = 0$ near an unperturbed global minimum $\hat{{\bf x}}\in\mathcal{X}^*$. Since $(\hat{{\bf x}}; 0)$ is a solution to the gradient equation and the Jacobian matrix with respect to ${\bf x}$ is equal to the positive definite Hessian $\nabla^2 f(\hat{{\bf x}})$, the Implicit Function Theorem (IFT) states that there exists a unique solution $\hat{{\bf x}}(\epsilon)$ in a neighbourhood of $\hat{{\bf x}}$ for all values of $\epsilon$ with a small norm. In addition, the continuity of Hessian implies that $\nabla_{\bf xx}\tilde{f}(\hat{\mathbf{x}}^*(\epsilon); \epsilon) \succ 0$. 
%
%
Thus, $\hat{\bf x}(\epsilon)$ is a local minimum of the perturbed problem \eqref{eqn:prob2}. As a result, we have proved the existence of a local minimum for the perturbed problem corresponding to each of the $2^{\lceil n/2\rceil}$ global minima of the unperturbed problem. Hence, the problem \eqref{eqn:prob2} has at least $2^{\lceil n/2 \rceil}$ local minima, while only two of them are global minima. In summary, we have constructed an instance in $\mathcal{L}(\mathcal{G}, n, 1)$ that has exponentially many spurious local solutions.
%
\end{example}

\begin{example}\label{exm:2}
Next, we construct an MC problem instance with exponentially many spurious local minima and $\mathcal{O}(n^2)$ observations. We choose the same ground truth matrix ${\bf M}^*(\epsilon)$ as in the last example,
 but assume that the measurement set $\Omega$ is induced by the graph $\mathcal{G}=(\mathcal{V},\mathcal{E}_1, \mathcal{E}_2)$ with $\mathcal{V} := [n]$, $\mathcal{E}_2 = \emptyset$ and
\begin{align*}
   & \mathcal{E}_1 := \left\{ (i,i),(i,2k),(2k,i)~|~\forall i\in[n],~ k\in\left[ \lfloor n/2 \rfloor \right] \right\}.
\end{align*}
%
Since the subgraph $\mathcal{G}_1 = (\mathcal{V}, \mathcal{E}_1)$ is connected and non-bipartite, the perturbed problem $\mathcal{P}_{{\bf M}^*(\epsilon), \Omega,n,1}$ belongs to the class $\mathcal{L}(\mathcal{G}, n, 1)$. Moreover, one can verify that the set of global minima of this problem is still $\mathcal{X}^*$ and the Hessian at every global solution is positive definite.
By the same argument, IFT implies that problem \eqref{eqn:prob2} has at least $2^{\lceil n/2 \rceil} -2$ spurious local minima for a generic and small perturbation $\epsilon$. 
%
\end{example}
%
Note that the instances analyzed in this section, as well as those in the remainder of this paper, satisfy the incoherence condition \citep{candes2009exact} with the parameter $\mu=\mathcal{O}(1)$.
The results in this subsection will be formalized with a unified framework next.

\subsection{Rank-$1$ Case with General Measurement Sets}
\label{sec:general-pattern}


In this subsection, we estimate the largest lower bound on the number of spurious local minima for the given parameters $n$ and $r$. We address the problem by first finding a lower bound on the number of spurious local minima given a general measurement set $\Omega$, and then maximizing the lower bound over $\Omega$. 
The following theorem utilizes the topology of $\mathcal{G}$ to quantify a lower bound on the number of spurious solutions for the measurement set $\Omega(\mathcal{G})$. 
%
\begin{theorem}\label{thm:spuriousmany}
Let $\mathcal{G}=(\mathcal{V}, \mathcal{E}_1, \emptyset)$ such that $\mathcal{G}_1=(\mathcal{V}, \mathcal{E}_1)$ is connected and non-bipartite with $n$ vertices. Assume that there exists a maximal independent set\footnote{For a graph $\mathcal{G}=(\mathcal{V},\mathcal{E})$, the set $\mathcal{S}\subset\mathcal{V}$ is called an independent set if no two nodes in $\mathcal{S}$ are adjacent. The set $\mathcal{S}$ is called a maximal independent set if it is an independent set and is not a strict subset of another independent set.} $\mathcal{S}(\mathcal{G}_1)$ of $\mathcal{G}_1$ such that every vertex in the set has a self-loop. There exists an instance in $\mathcal{L}(\mathcal{G}, n, 1)$ for which the problem \eqref{eqn:obj-linear} has at least $2^{|\mathcal{S}(\mathcal{G}_1)|}-2$ spurious local minima.
\end{theorem}
In both Examples \ref{exm:1} and \ref{exm:2}, a maximal independent set is $\mathcal{S}=\{2k+1~|~ \forall k\in[\lceil n/2 \rceil]\} ~$\footnote{$\mathcal{S}$ is shorthand notation for $\mathcal{S}(\mathcal{G}_1)$.}. Hence, Theorem \ref{thm:spuriousmany} implies that there are $2^{\lceil n/2 \rceil} -2$ spurious local minima, which is consistent with our analysis.
Since a maximal independent set of a connected and non-bipartite graph can have up to $n-1$ vertices, the number of spurious local minima can be as large as $2^{n-1}-2$. 
\begin{corollary}\label{cor:singlemissing}
There exist a graph $\mathcal{G}$ and an instance in $\mathcal{L}(\mathcal{G}, n, 1)$ such that problem \eqref{eqn:obj-linear} has $2^{n-1}-2$ spurious solutions. In addition, there exist a graph $\mathcal{G}$ and an instance in $\mathcal{L}(\mathcal{G}, n, 1)$ with $|\Omega| = n^2 - 2$ such that the problem \eqref{eqn:obj-linear} has spurious solutions.
\end{corollary}
Corollary \ref{cor:singlemissing} implies that the B-M factorization may not be an efficient approach to the MC problem, since it has a spurious solution even in the highly ideal case when almost all entries of the matrix are measured. Generally, the proof of Theorem \ref{thm:spuriousmany} implies that, as a necessary condition for not having a spurious solution in formulation \eqref{eqn:obj-linear}, the elements of ${\bf x}^*$ associated with the nodes outside of the maximal independent set $\mathcal{S}$ should not be much smaller than those associated with the nodes in $\mathcal{S}$. 
\begin{corollary}\label{cor:function}
Under the setting of Theorem \ref{thm:spuriousmany}, there exists a function $h_{\mathcal{S}}(\cdot):(0,\infty)\mapsto(0,\infty)$ such that $\mathcal{P}_{{\bf x}^*({\bf x}^*)^T, \Omega(\mathcal{G}), n, 1}$ has at least $2^{|\mathcal{S}|}-2$ spurious local minima in formulation \eqref{eqn:obj-linear} for every generic ${\bf x}^*\in\mathbb{R}^n$ satisfying 
\[ \setlength{\abovedisplayskip}{3pt}\setlength{\belowdisplayskip}{3pt}\|\mathbf{x}^*_{\mathcal{S}^c}\| \leq h_{\mathcal{S}}( {\min}_{i \in \mathcal{S}} |x^*_i|) \cdot \|{\bf x}^*_\mathcal{S}\|, \]
where $\mathcal{S}^c:=[n]\backslash\mathcal{S}$ and ${\bf x}_{\mathcal{S}^c}:=(x_i : i\notin \mathcal{S})$.
\end{corollary}
Because the maximal independent set of a graph $\mathcal{G}$ is not necessarily unique, the set of functions $h_{\mathcal{S},M}(\cdot)$ over all maximal independent sets $\mathcal{S}$ designates a necessary condition for the nonexistence of spurious local minima given a measurement set $\Omega(\mathcal{G})$.

\subsection{Extension to General Rank-$r$ Case}
\label{sec:rankr}

We generalize the results to the case when the ground truth matrix has a general rank. \cite{eisenberg2013complexity} showed that the rank-$r$ MC problem is $\mathcal{NP}$-hard in the worst case for every $r \ge 2$. However, we focus on instances in the low-complexity class $\mathcal{L}(\mathcal{G}, n, r)$ and show that there are instances in this class whose B-M factorization formulation \eqref{eqn:obj-linear} has a highly undesirable landscape. We cannot simply extend the proof of the rank-$1$ case to the rank-$r$ case since there exist an infinite number of matrices ${\bf X}^*$ such that ${\bf M}^*={\bf X}^*({\bf X}^*)^T$ when $r >1$. The global optimality of a solution ${\bf X}^*$ is not lost under any orthogonal transformation.
This implies that the Hessian at the global solutions of problem \eqref{eqn:obj-linear} cannot be positive definite, which fails the applicability of IFT. Therefore, we consider the quotient manifold $\mathbb{R}^{n\times r}/O_r$, where $O_r$ is the lie group of $r\times r$ orthogonal matrices. To simplify the analysis, we instead consider the following lower-diagonal subspace
\begin{equation*}
\setlength{\abovedisplayskip}{3pt}\setlength{\belowdisplayskip}{3pt}
\begin{aligned}
    \mathcal{W}^{n \times r} := \{ \mathbf{X} \in \mathbb{R}^{n \times r}~|~ \mathbf{X}_{ij} = 0,~ \forall i\in[n]&,~ j\in[r]\\
    &\mathrm{s.t.}\quad i<j \}.
\end{aligned}
\end{equation*}
%
%
We define an embedding of the manifold $\mathbb{R}^{n\times r}/O_r$ into $\mathcal{W}^{n \times r}$ and composite it with the quotient map.
\begin{definition}[Restriction map] \label{def:orthrest}
Given a matrix ${\bf X} \in \mathbb{R}^{n\times r}$, we define the embedding $\phi^{emb}([{\bf X}]) := {\bf R}\in\mathcal{W}^{n\times r}$, where ${\bf X} = {\bf RQ}$ is the RQ decomposition with $\mathbf{Q}$ being an orthogonal matrix and ${\bf R}$ having non-negative diagonal elements.
The restriction map is defined as $\phi({\bf X}) := \phi^{emb}([{\bf X}])$. 
\end{definition}
%
When the RQ decomposition is not unique, we choose an arbitrary decomposition for the embedding $\phi^{emb}(\cdot)$. However, the properties of the RQ decomposition ensure that $\phi^{emb}(\cdot)$ is a bijection in a small neighborhood of each matrix ${\bf X}$ whose first $r$ rows are linearly independent. 
%
Consider the restricted version of problem \eqref{eqn:obj-linear}:
\begin{equation}\label{eqn:probres}
\setlength{\abovedisplayskip}{3pt}\setlength{\belowdisplayskip}{3pt}
\min_{\mathbf{X} \in \mathcal{W}^{n \times r} } f(\mathbf{X}),
\end{equation}
%
Results in Section \ref{sec:general-pattern} can be extended to the problem \eqref{eqn:probres}, and then be translated back to the problem \eqref{eqn:obj-linear}.
\begin{lemma}\label{lem:rankr-1}
Consider a graph $\mathcal{G} = (\mathcal{V}, \mathcal{E}_1, \mathcal{E}_2)$ with $m=\lfloor n/r\rfloor$ vertices for which the subgraph $\mathcal{G}_1 = (\mathcal{V}, \mathcal{E}_1)$ is connected and non-bipartite. Assume that there exists a maximal independent set $\mathcal{S}(\mathcal{G}_1)$ of $\mathcal{G}_1$ whose vertices each have a self-loop. If the induced subgraph\footnote{See \cite{harary2018graph} for the definition.} $\mathcal{G}_2[\mathcal{S}]$ is connected, then there exists an instance in $\mathcal{L}(\mathcal{G}, n, r)$ for which the problem \eqref{eqn:probres} has at least $2^{r|\mathcal{S}(\mathcal{G}_1)|}-2^r$ spurious local minima. In addition, the first $r$ rows of each local minimum are linearly independent.
\end{lemma}
The Hessian at each local minimum of the unperturbed problem is positive definite along the tangent space of $\mathcal{W}^{n\times r}$, where the off-diagonal observations of $\Omega(\mathcal{G})$ play a key role. If the first $r$ rows of a local minimum $\hat{\bf X}$ are linearly independent, the diagonal elements of $\hat{\bf X}$ are nonzero. Therefore, by flipping the signs of columns, we can find an equivalent local minimum $\tilde{\bf X}$ with positive diagonal elements, i.e., $\tilde{\bf X}$ lies in the range of $\phi^{emb}(\cdot)$. By symmetry, the total number of such local minima is $2^{r(|\mathcal{S}(\mathcal{G}_1)|-1)}-1$. Since the restriction map $\phi^{emb}(\cdot)$ is a bijection in a neighborhood of $\tilde{\bf X}$, the equivalent class $[ \tilde{\bf X} ]\in\mathbb{R}^{n\times r}/O_r$ is a local minimum of problem \eqref{eqn:obj-linear} on the quotient manifold and thus $\tilde{\bf X}$ is a local minimum of problem \eqref{eqn:obj-linear}. The above argument leads to Theorem \ref{thm:rspuriousmany}.
%
\begin{theorem}\label{thm:rspuriousmany}
Consider a graph $\mathcal{G} = (\mathcal{V}, \mathcal{E}_1, \mathcal{E}_2)$ satisfying the conditions of Lemma \ref{lem:rankr-1}. There exists an instance in $\mathcal{L}(\mathcal{G}, n, r)$ for which problem \eqref{eqn:obj-linear} has at least $2^{r(|\mathcal{S}(\mathcal{G}_1)|-1)}-1$ equivalent classes of spurious solutions.
\end{theorem}
Finally, we give an estimate on the largest lower bound for the number of spurious local minima.
\begin{corollary}\label{cor:rsinglemissing}
There exists an instance in $\mathcal{L}(\mathcal{G}, n, r)$ for which the problem \eqref{eqn:obj-linear} has at least $ 2^{n - 2r} - 1$ equivalent classes of spurious solutions. In addition, there exists an instance in $\mathcal{L}(\mathcal{G}, n, r)$ with $|\Omega| = n^2-2r$ for which the problem \eqref{eqn:obj-linear} has spurious solutions.
\end{corollary}

\subsection{General Loss Functions}

In this part, we generalize the preceding results to the problem \eqref{eqn:obj-general}. To extend the constructions to a general loss function $g(\cdot)$, we require a few weak assumptions on the loss function $g(\cdot)$.
%
\begin{assumption}\label{asp:general}
The following conditions hold for the function $g(\cdot)$:
\setlist{nolistsep}
\begin{enumerate}[label= \bf \roman*), noitemsep]
    \item $g(\cdot)$ is twice continuously differentiable;
    \item the matrix $\mathbf{0}_{n\times r}$ is the unique minimizer of $g(\cdot)$;
    \item the Hessian of $g(\cdot)$ at $\mathbf{0}_{n\times r}$ is positive definite.
\end{enumerate}
\end{assumption}
%
Now, we can extend the results in Section \ref{sec:rankr} to the general loss function case under the above assumption.
\begin{theorem}\label{thm:characterization}
Consider a graph $\mathcal{G} = (\mathcal{V}, \mathcal{E}_1, \mathcal{E}_2)$ satisfying the conditions of Lemma \ref{lem:rankr-1} and suppose that Assumption \ref{asp:general} holds. There exists an instance in $\mathcal{L}(\mathcal{G}, n, r)$ for which the problem \eqref{eqn:obj-general} has at least $2^{r(|\mathcal{S}(\mathcal{G}_1)|-1)}-1$ equivalent classes of spurious local minima.
\end{theorem}
We note that the $\ell_2$-loss function $g(\cdot) = \|\cdot\|_F^2$ satisfies the conditions in Assumption \ref{asp:general}. As another example, regularizers are ubiquitously used in the low-rank matrix optimization literature \citep{ge2016matrix, ge2017no,fattahi2020exact}. As a corollary to Theorem \ref{thm:characterization}, the regularized version of the problem \eqref{eqn:obj-linear} also suffers from the same issue. In this case, the loss function is equal to
\[\setlength{\abovedisplayskip}{3pt} g({\bf X}) := \left\|(\mathbf{XX}^T - {\bf M}^*)_\Omega\right\|_F^2 + Q({\bf X}), \setlength{\belowdisplayskip}{3pt}\]
where $Q({\bf X}) := \lambda \left( \|{\bf X}_i\| - \alpha \right)_+^4$ is from \cite{ge2016matrix}, $(x)_+:=\max\{x,0\}$ and $\alpha,\lambda > 0$ are constants.
Since the regularizer does not change the 
landscape around global solutions with a large $\alpha$, Assumption \ref{asp:general} is satisfied and Theorem \ref{thm:characterization} is applicable to the regularized problem.

\section{MORE OBSERVATIONS LEAD TO SPURIOUS LOCAL MINIMA}
\label{sec:more-obs}

In Section \ref{sec:exp-sol}, we showed that the B-M factorization formulation \eqref{eqn:obj-linear} has an exponential number of spurious local minima on low-complexity MC problem instances. 
In this section, we exhibit another unfavorable behaviour of the B-M factorization. We identify an MC problem instance in $\mathcal{L}(\mathcal{G}, n, r)$ with some pattern $\Omega$ that has no spurious solution while adding 
observations to $\Omega$ leads to spurious solutions. 
Let $m$ and $r$ be natural numbers with $m\geq 2r$. We define $n := mr$ and let the graph be $\mathcal{G}_k := (\mathcal{V},\mathcal{E}_k)$ where $k\in[m]$ is an arbitrary index and
\begin{equation*}
\setlength{\abovedisplayskip}{3pt}\setlength{\belowdisplayskip}{3pt}
    \mathcal{V} := [m],\quad \mathcal{E}_k := \{ (k,j),~(j,k) ~|~ \forall j \in [m] \}.
\end{equation*}
In the measurement set $\Omega_k$, we observe the blocks ${\bf M}_{i,j}$ for all $(i,j)\in\mathcal{E}_k$;
%
see Figure \ref{fig:matrixprojection}.
\begin{figure}[t]
    \centering
    \begin{equation*}
    \resizebox{.95\hsize}{!}{ 
    $\mathbf{M}_{\Omega_k} =  \left[ \begin{array}{c|c|c}
        \mathbf{0} & \begin{array}{c}
            \mathbf{M}_{1,k}  \\ 
            \vdots  
        \end{array}  & \mathbf{0} \\ 
        \hline
        \begin{array}{cc}
            \mathbf{M}_{k,1} & \cdots  \\
        \end{array} & \mathbf{M}_{k,k} & \begin{array}{cc}
             \cdots & \mathbf{M}_{k,m}
        \end{array} \\
        \hline 
        \mathbf{0} & \begin{array}{c}
            \vdots  \\ 
            \mathbf{M}_{m,k}
        \end{array}  & \mathbf{0}
    \end{array} \right]$
    }
    \end{equation*}
    \caption{Measurement operator $\mathcal{A}_{\Omega_k}$ on matrix $\mathbf{M}$.}
    \label{fig:matrixprojection}
    \vspace{-1.5em}
\end{figure}
$\Omega_k$ contains only full block observations induced by $\mathcal{G}_k$. The next proposition states that if the ground truth matrix is generic, every SOCP\footnote{Second-order
critical points are defined as those points that satisfy the first-order and the second-order necessary optimality conditions.} of the problem \eqref{eqn:obj-linear} is a global minimum.
\begin{proposition}\label{prop:nospurious}
Given an index $k\in[m]$, let the measurement set $\Omega$ be equal to ${\Omega}_k$. Assume that the block ${\bf M}^*_{i,j}$ of the ground truth matrix ${\bf M}^*$ has rank $r$ for all $i,j\in[m]$. Then, every 
SOCP of problem \eqref{eqn:obj-linear} is a global minimum.
\end{proposition}
 %


%
Next, we construct a graph $\tilde{\mathcal{G}}_k := (\mathcal{V}, \tilde{\mathcal{E}}_k, \mathcal{E}_2) $, where
\begin{equation*}
\setlength{\abovedisplayskip}{3pt}\setlength{\belowdisplayskip}{3pt}
\begin{aligned}
\tilde{\mathcal{E}}_k &:= \mathcal{E}_k \cup \{ (i,i)~|~\forall i\in[m] \},\\
\mathcal{E}_2 &:= \{(i,j)~|~ \forall i,j \in[m], ~i\neq j\}.
\end{aligned}
\end{equation*}
Namely, we have included all self-loops and nondiagonal observations of each block in the new graph. A maximal independent set for the subgraph $\tilde{\mathcal{G}}_{k,1} := (\mathcal{V}, \tilde{\mathcal{E}}_k) $  is $\mathcal{S} := [m] \backslash \{k\}$. We define a new measurement set $\tilde{\Omega}_k  := \Omega( \tilde{\mathcal{G}}_k)$. Since $\tilde{\mathcal{E}}_k$ is a superset of $\mathcal{E}_k$, the measurement set $\tilde{\Omega}_k$ is larger than $\Omega_k$. Using Theorem \ref{thm:rspuriousmany}, we obtain the following result.
%
%
\begin{corollary}\label{cor:moreobs}
Every instance of the MC problem with the measurement set $\tilde{\Omega}_k$ and full rank blocks ${\bf M}^*_{i,j}$ for all $i,j\in[m]$ belongs to $\mathcal{L}(\mathcal{G}, n, r)$. The formulation \eqref{eqn:obj-linear} of an instance of the problem has at least $2^{r(m-2)}-1$ equivalent classes of spurious local minima, while all spurious solutions disappear when using the smaller set $\Omega_k$.
\end{corollary}
Results of Proposition \ref{prop:nospurious} and Corollary \ref{cor:moreobs} conclude that the landscape of the problem \eqref{eqn:obj-linear} deteriorates when the number of observations is increased. This phenomenon further reveals the unfavorable behavior of the B-M factorization on low-rank matrix optimization problems, even if the information-theoretic complexity is low.

\section{MEASURE OF COMPLEXITY FOR FACTORIZATION APPROACH}
\label{sec:metric}

In Section \ref{sec:exp-sol}, we showed that if there is an MC problem with a non-unique completion, a slightly perturbed problem will have exponentially many spurious local minima in the B-M factorization formulation \eqref{eqn:obj-general}.
%
Hence, bifurcation behaviors appear around measurement matrices ${\bf M}^*_\Omega$ that are associated with multiple global solutions. For a given measurement operator $\mathcal{A}$, a measurement matrix $\mathcal{A}({\bf M})$ that allows multiple global solutions designates an unacceptable region in the space of ground truth solutions.
Based on this intuition, we define a metric to capture the extent of the bifurcation behavior. We define the set of measurements that allow non-unique solutions:
\begin{equation*}
\setlength{\abovedisplayskip}{3pt}\setlength{\belowdisplayskip}{3pt}
    \begin{aligned}
        \mathcal{T}_\mathcal{A} := \Big\{ \mathcal{A}({\bf M})&:~ \exists \mathbf{X}_1, {\bf X}_2\in\mathbb{R}^{n\times r} \\
        \mathrm{ s.t. }~ &\mathbf{X}_1{\bf X}_1^T \neq \mathbf{X}_2{\bf X}_2^T,\\
        &\mathcal{A}({\bf M}) = \mathcal{A}(\mathbf{X_1X_1}^T) = \mathcal{A}(\mathbf{X_2X_2}^T) \Big\}.
    \end{aligned}
\end{equation*}
Then, we define a complexity metric below.
\begin{definition}[Complexity metric]
The complexity metric for operator $\mathcal{A}$ and ground truth ${\bf M}^*$ is defined as
\begin{equation*}
\setlength{\abovedisplayskip}{3pt}\setlength{\belowdisplayskip}{3pt}
    \mathrm{dist}(\mathcal{A}({\bf M}^*), \mathcal{T}_\mathcal{A}) := \min_{\mathcal{A}({\mathbf{M}}) \in \mathcal{T}_\mathcal{A}}  \|\mathcal{A}(\mathbf{M}^*) - \mathcal{A}({\mathbf{M}}) \|_F.
\end{equation*}
\end{definition}
%

It is expected that for instances with a large complexity metric, the optimization complexity of algorithms based on the B-M factorization approach will be aligned with the corresponding information-theoretic complexity.
%
For example, if the RIP condition is satisfied, the set $\mathcal{T}$ is empty and therefore 
the complexity metric is always $\infty$. Oppositely, the instances studied in Section \ref{sec:exp-sol} with spurious solutions all have small complexity metrics. 
Consequently, the complexity metric is a possible measure of the optimization complexity for the MC problem with the B-M factorization: the optimization problem should be more difficult if the complexity metric is lower.

\section{GRADIENT-BASED METHODS FAIL WITH HIGH PROBABILITY}
\label{sec:gd}

We show that the exponential number of spurious local minima in preceding instances will make most randomly initialized gradient-based methods fail with a high probability. The existence of spurious local minima does not necessarily imply the failure of gradient-based methods; see \cite{ma2018gradient,chen2020nonconvex}. The analysis in this section is based on the gradient flow
\begin{equation} \label{eqn:grad-flow}
\setlength{\abovedisplayskip}{3pt}\setlength{\belowdisplayskip}{3pt}
\dot{{\bf X}}(t) = - \nabla_{\bf X} f({\bf X}(t)),\quad {\bf X}(0) = {\bf X}_0.
\end{equation}
%
It is known that the trajectories of gradient-based methods with a small enough step size are close to those of the gradient flow. We can view a variety of gradient-based methods as ordinary differential equation (ODE) solvers applied to the gradient flow \eqref{eqn:grad-flow}. Then, the convergence of the discrete trajectories when the step size goes to $0$ can be guaranteed by the consistency and the stability of the ODE solver. \cite{scieur2017integration}  proved that the consistency and stability conditions are satisfied by several commonly used gradient-based methods, such as the gradient descent, the proximal point and the accelerated gradient descent methods. Although \cite{scieur2017integration} considered minimizing a strongly convex function, the consistency and stability conditions only depend on the ODE solver and the Lipschitz continuity of the underlying gradient flow.
%
We need the following assumption on the loss function to characterize the global landscape.
\begin{assumption}\label{asp:gd-1}
The loss function $g(\cdot)$ satisfies the sparse $(\delta,r)$-RIP condition in the $\Omega$-norm for some constant $\delta\in[0,1)$ and integer $r\geq 1$. Namely, the inequality
\begin{align*}
    (1-\delta) \| {\bf N}_\Omega\|_F^2 \leq {\bf N} : \nabla^2 g({\bf M}_{\Omega}) : {\bf N} \leq (1+\delta) \| {\bf N}_\Omega\|_F^2
\end{align*}
holds for all matrices ${\bf M}$ and ${\bf N}$ with rank at most $2r$.
\end{assumption}
The sparse RIP condition is remarkably different from the conservative RIP condition. For example, the $\ell_2$-loss function satisfies the sparse $(0,r)$-RIP condition for every $r$, while the RIP condition does not hold if $\Omega$ is not a complete graph. Under the above assumption, we show that for the unperturbed example constructed in Section \ref{sec:exp-sol}, the gradient flow will converge to each global minimum with equal probability in the rank-$1$ case. The main difficulty is to show that all saddle points of the objective function $f({\bf X})$ are strict and therefore their region of attractions (ROAs) have measure zero \citep{lee2016gradient}.
\begin{lemma}\label{lem:gd-1}
Suppose that Assumption \ref{asp:gd-1} holds for $r=1$ and $\delta=\mathcal{O}(1/n)$, where $n\geq 3$ is the size of the ground truth matrix. There exists an MC problem instance such that the following statements hold for the problem \eqref{eqn:obj-general}:
\setlist{nolistsep}
\begin{itemize}[noitemsep]
    \item there are $2^{\lceil n/2\rceil}$ equivalent global minima;
    \item if the gradient flow \eqref{eqn:grad-flow} is initialized with an absolutely continuous radial probability distribution, it converges to each global minimum with the equal probability $2^{-\lceil n/2\rceil}$,
\end{itemize}
where a probability distribution is called radial if its density function at point $x$ only depends on $\|x\|$ and the absolute continuity is with respect to the Lebesgue measure.
\end{lemma}
Examples of absolutely continuous radial probability distributions include zero-mean Gaussian distributions and uniform distributions over a ball centered at the origin. Note that the $\ell_2$-loss function satisfies the assumption of Lemma \ref{lem:gd-1}. Next, we show that with a sufficiently small perturbation to the previous instance of the problem, the ROA of each local minimum will not shrink significantly. Therefore, the gradient flow will converge to each global minimum or spurious solution with approximately the same probability.
\begin{theorem}\label{thm:gd-1}
Under the setting of Lemma \ref{lem:gd-1}, consider an absolutely continuous radial probability distribution. There exists an instance in $\mathcal{L}(\mathcal{G}, n, 1)$ for which the problem \eqref{eqn:obj-general} satisfies the following properties:
\setlist{nolistsep}
\begin{itemize}[noitemsep]
    \item the global minima are unique up to a sign flip;
    \item if the gradient flow \eqref{eqn:grad-flow} is initialized with the given distribution, it converges to a global minimum associated with the ground truth solution with probability at most $\mathcal{O}(2^{-\lceil n/2\rceil})$.
\end{itemize}
\end{theorem}
%
The results of Theorem \ref{thm:gd-1} imply that, in the rank-$1$ case, most gradient-based methods with a small enough step size and a suitable random initialization will converge to a spurious solution with an overwhelming probability. The proof works for the general rank case if it can be shown that there is no degenerate saddle points for the above-mentioned unperturbed instances of the problem. 

%
\begin{remark} We remark that the trajectories of stochastic gradient descent (SGD) methods cannot be approximated by those of the gradient flow. Hence, our analysis cannot automatically imply the failure of SGD. However, the proof of Lemma \ref{lem:gd-1} can be adopted to conclude that SGD methods with a random initialization will converge to each global minimum with equal probability in the unperturbed case. Since the trajectories of SGD methods will not vary dramatically with a sufficiently small perturbation, they still converge to each solution with approximately the same probability. Therefore, we also expect the SGD methods to fail with high probability. 
\end{remark}

\section{EXPERIMENTS}
\label{sec:num}

Numerical results are presented to support the failure of the gradient descent algorithm. Each MC problem with the B-M factorization formulation \eqref{eqn:obj-linear} is solved by the gradient descent algorithm with a constant step size, where the step size is chosen to be small enough to guarantee that the algorithm converges to a stationary point. Regarding the measurement set, the graph $\mathcal{G}_1:=(\mathcal{V},\mathcal{E}_1)$ is generated randomly by the Erd\"{o}s–R\'{e}nyi model $G(m,p)$, where $\mathcal{V}:=[m]$ and each edge of the graph is included independently with probability $p$. If $\mathcal{G}_1$ is not connected or a node in the maximal independent set $\mathcal{S}$ does not have a self-loop, the missed edges are added to satisfy these conditions. In addition, a connected subtree $\mathcal{G}_2 = (\mathcal{S}, \mathcal{E}_2(\mathcal{S}))$ is generated for nondiagonal observations. We define $\mathcal{G}:=(\mathcal{V},\mathcal{E}_1,\mathcal{E}_2)$ and subsequently the measurement set $\Omega(\mathcal{G})$. In addition, the unperturbed ground truth matrix $\mathbf{M}^*={\bf X^*}({\bf X}^*)^T$ is defined as $\mathbf{M}^*_{i,j} := \mathbf{I}_r$ for all $i, j \in \mathcal{S}$, and $\mathbf{M}^*_{i,j} := \mathbf{0}_{r\times r}$ otherwise. Lastly, a Gaussian random perturbation matrix $\epsilon\in\mathbb{R}^{n\times r}$ is generated and normalized, e.g. $\|\epsilon\|_F =1$. Then, the perturbed ground truth matrix $\mathbf{M}^*(\epsilon) := (\mathbf{X^* + \gamma\epsilon})(\mathbf{X^* + \gamma\epsilon})^T $ is generated, where $\gamma>0$ is the perturbation size. We evaluate the success rate of the algorithm at $100$ equally distributed values of $\gamma \in (0, 0.5)$ with $300$ random initializations of the gradient algorithm for each instance and each $\gamma$. 
\begin{figure}
    \centering
    \includegraphics[scale =0.25]{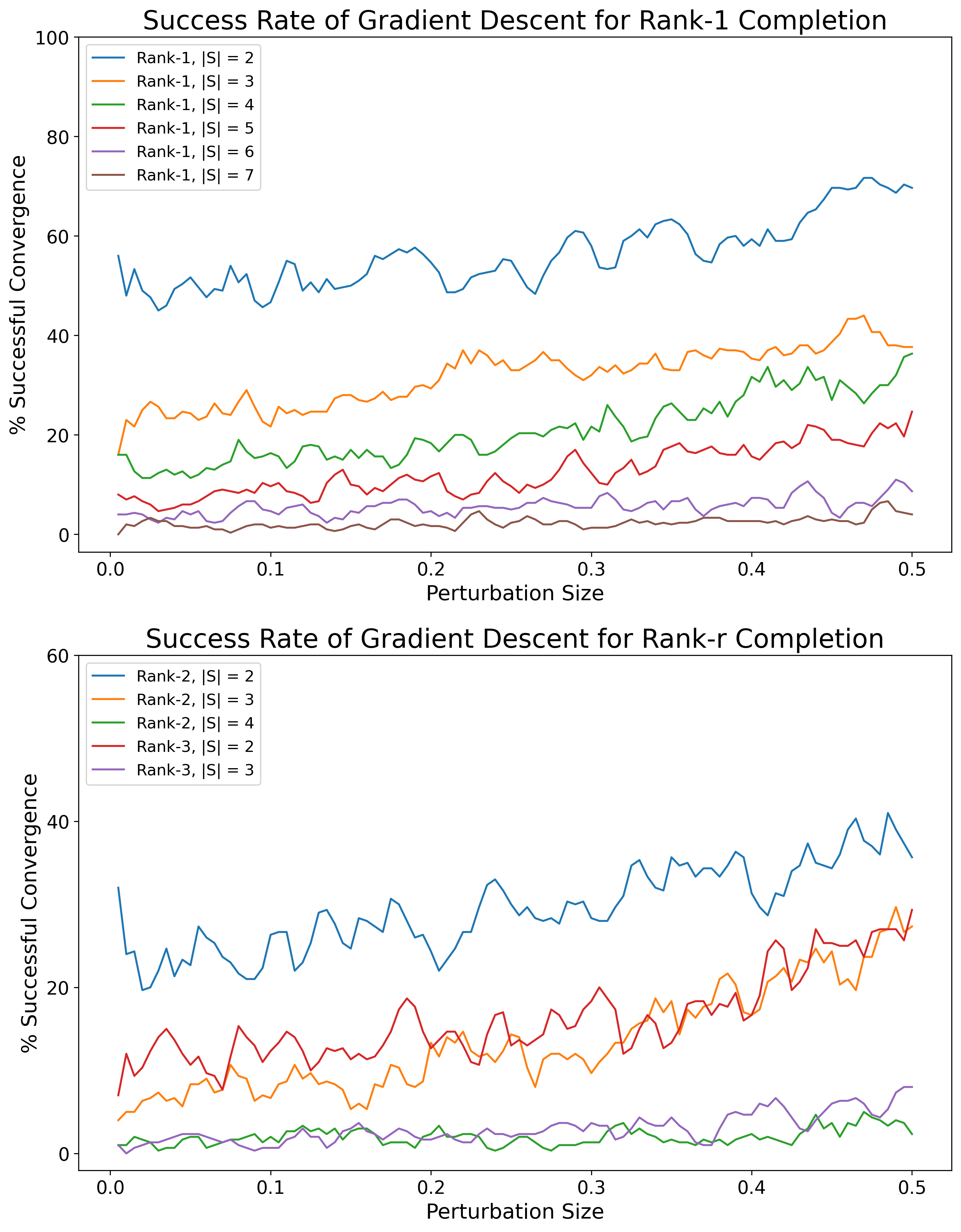}
    \caption{Success rate of gradient descent method for \textit{(top)} rank-1 and \textit{(bottom)} rank-r MC problem instances with randomly generated measurement sets.}
    \label{fig:experiments}
    \vspace{-1.5em}
\end{figure}

The top figure in Figure \ref{fig:experiments} illustrates the success rate of the gradient descent algorithm with the rank $r=1$, dimension $n = 20$ and various maximum independent set sizes $|\mathcal{S}|$. The results conform to Theorem \ref{thm:gd-1} implying that the success rate is less than $1/(2^{|\mathcal{S}|-1})$ when the perturbation size $\gamma$ is small. The bottom figure in Figure \ref{fig:experiments} makes similar observations with different ranks and maximum independent set sizes $|\mathcal{S}|$ when $m$ is equal to $10$. These observations imply that Theorem \ref{thm:gd-1} can be extended to the general rank case, since the success rate is less than $1/(2^{r(|\mathcal{S}|-1)})$ when $\gamma$ is small. We note that the behavior of the algorithm may change when $\gamma$ is large.
Specifically, significant improvements in success rate are observed when $\gamma > 0.2$ for most problem instances. This is in accordance with our notion of complexity metric.

\section{CONCLUSION}

In this paper, we provided a negative answer to the question of whether the B-M factorization approach can capture the benign properties of low-complexity MC problem instances. More specifically, we defined a class of MC problem instances that could be solved in polynomial time. We showed that there exist MC problem instances in this class that have exponentially many spurious local minima in the B-M factorization formulation \eqref{eqn:obj-general}. 
The results hold for a general class of loss functions, including the commonly used regularized formulation. In addition, for the rank-$1$ case, we proved that gradient-based methods fail with high probability for such instances. Numerical results verify that similar behaviors also hold for higher rank cases. These results imply that the optimization complexity of methods based on the factorized problem \eqref{eqn:obj-general} are not aligned with the information-theoretic complexity of the MC problem. Furthermore, we derived a complexity metric that potentially captures the complexity of the B-M factorization formulation \eqref{eqn:obj-general}.

\bibliographystyle{icml2021}
\bibliography{ref}

\begin{thebibliography}{29}
\providecommand{\natexlab}[1]{#1}
\providecommand{\url}[1]{\texttt{#1}}
\expandafter\ifx\csname urlstyle\endcsname\relax
  \providecommand{\doi}[1]{doi: #1}\else
  \providecommand{\doi}{doi: \begingroup \urlstyle{rm}\Url}\fi

\bibitem[Bhojanapalli \& Jain(2014)Bhojanapalli and
  Jain]{bhojanapalli2014universal}
Bhojanapalli, S. and Jain, P.
\newblock Universal matrix completion.
\newblock In \emph{International Conference on Machine Learning}, pp.\
  1881--1889. PMLR, 2014.

\bibitem[Bhojanapalli et~al.(2016)Bhojanapalli, Kyrillidis, and
  Sanghavi]{bhojanapalli2016dropping}
Bhojanapalli, S., Kyrillidis, A., and Sanghavi, S.
\newblock Dropping convexity for faster semi-definite optimization.
\newblock In \emph{Conference on Learning Theory}, pp.\  530--582. PMLR, 2016.

\bibitem[Bi \& Lavaei(2021)Bi and Lavaei]{bi2020global}
Bi, Y. and Lavaei, J.
\newblock On the absence of spurious local minima in nonlinear low-rank matrix
  recovery problems.
\newblock In \emph{International Conference on Artificial Intelligence and
  Statistics}, pp.\  379--387. PMLR, 2021.

\bibitem[Bi et~al.(2021)Bi, Zhang, and Lavaei]{bi2021local}
Bi, Y., Zhang, H., and Lavaei, J.
\newblock Local and global linear convergence of general low-rank matrix
  recovery problems.
\newblock \emph{arXiv preprint arXiv:2104.13348}, 2021.

\bibitem[Burer \& Monteiro(2003)Burer and Monteiro]{burer2003nonlinear}
Burer, S. and Monteiro, R.~D.
\newblock A nonlinear programming algorithm for solving semidefinite programs
  via low-rank factorization.
\newblock \emph{Mathematical Programming}, 95\penalty0 (2):\penalty0 329--357,
  2003.

\bibitem[Cand{\`e}s \& Recht(2009)Cand{\`e}s and Recht]{candes2009exact}
Cand{\`e}s, E.~J. and Recht, B.
\newblock Exact matrix completion via convex optimization.
\newblock \emph{Foundations of Computational mathematics}, 9\penalty0
  (6):\penalty0 717--772, 2009.

\bibitem[Cand{\`e}s \& Tao(2010)Cand{\`e}s and Tao]{candes2010power}
Cand{\`e}s, E.~J. and Tao, T.
\newblock The power of convex relaxation: Near-optimal matrix completion.
\newblock \emph{IEEE Transactions on Information Theory}, 56\penalty0
  (5):\penalty0 2053--2080, 2010.

\bibitem[Chen et~al.(2020)Chen, Liu, and Li]{chen2020nonconvex}
Chen, J., Liu, D., and Li, X.
\newblock Nonconvex rectangular matrix completion via gradient descent without
  $\ell_{2,\infty}$ regularization.
\newblock \emph{IEEE Transactions on Information Theory}, 66\penalty0
  (9):\penalty0 5806--5841, 2020.

\bibitem[Chi et~al.(2019)Chi, Lu, and Chen]{chi2019nonconvex}
Chi, Y., Lu, Y.~M., and Chen, Y.
\newblock Nonconvex optimization meets low-rank matrix factorization: An
  overview.
\newblock \emph{IEEE Transactions on Signal Processing}, 67\penalty0
  (20):\penalty0 5239--5269, 2019.

\bibitem[Davenport et~al.(2014)Davenport, Plan, Van Den~Berg, and
  Wootters]{davenport20141}
Davenport, M.~A., Plan, Y., Van Den~Berg, E., and Wootters, M.
\newblock 1-bit matrix completion.
\newblock \emph{Information and Inference: A Journal of the IMA}, 3\penalty0
  (3):\penalty0 189--223, 2014.

\bibitem[Eisenberg-Nagy et~al.(2013)Eisenberg-Nagy, Laurent, and
  Varvitsiotis]{eisenberg2013complexity}
Eisenberg-Nagy, M., Laurent, M., and Varvitsiotis, A.
\newblock Complexity of the positive semidefinite matrix completion problem
  with a rank constraint.
\newblock \emph{Springer, Fields Institute Communications, vol. 69}, 2013.

\bibitem[Fattahi \& Sojoudi(2020)Fattahi and Sojoudi]{fattahi2020exact}
Fattahi, S. and Sojoudi, S.
\newblock Exact guarantees on the absence of spurious local minima for
  non-negative rank-1 robust principal component analysis.
\newblock \emph{Journal of machine learning research}, 2020.

\bibitem[Ge et~al.(2016)Ge, Lee, and Ma]{ge2016matrix}
Ge, R., Lee, J.~D., and Ma, T.
\newblock Matrix completion has no spurious local minimum.
\newblock \emph{Advances in Neural Information Processing Systems}, pp.\
  2981--2989, 2016.

\bibitem[Ge et~al.(2017)Ge, Jin, and Zheng]{ge2017no}
Ge, R., Jin, C., and Zheng, Y.
\newblock No spurious local minima in nonconvex low rank problems: A unified
  geometric analysis.
\newblock In \emph{International Conference on Machine Learning}, pp.\
  1233--1242. PMLR, 2017.

\bibitem[Harary(2018)]{harary2018graph}
Harary, F.
\newblock Graph theory.
\newblock \emph{CRC Press}, 2018.

\bibitem[Khalil(2002)]{Khalil:1173048}
Khalil, H.~K.
\newblock \emph{{Nonlinear systems; 3rd ed.}}
\newblock Prentice-Hall, Upper Saddle River, NJ, 2002.

\bibitem[Kir{\'a}ly et~al.(2015)Kir{\'a}ly, Theran, and
  Tomioka]{kiraly2015algebraic}
Kir{\'a}ly, F.~J., Theran, L., and Tomioka, R.
\newblock The algebraic combinatorial approach for low-rank matrix completion.
\newblock \emph{J. Mach. Learn. Res.}, 16\penalty0 (1):\penalty0 1391--1436,
  2015.

\bibitem[Lee et~al.(2016)Lee, Simchowitz, Jordan, and Recht]{lee2016gradient}
Lee, J.~D., Simchowitz, M., Jordan, M.~I., and Recht, B.
\newblock Gradient descent only converges to minimizers.
\newblock In \emph{Conference on learning theory}, pp.\  1246--1257. PMLR,
  2016.

\bibitem[Li et~al.(2016)Li, Liang, and Risteski]{li2016recovery}
Li, Y., Liang, Y., and Risteski, A.
\newblock Recovery guarantee of weighted low-rank approximation via alternating
  minimization.
\newblock In \emph{International Conference on Machine Learning}, pp.\
  2358--2367. PMLR, 2016.

\bibitem[Ma et~al.(2019)Ma, Wang, Chi, and Chen]{ma2019implicit}
Ma, C., Wang, K., Chi, Y., and Chen, Y.
\newblock Implicit regularization in nonconvex statistical estimation: Gradient
  descent converges linearly for phase retrieval, matrix completion, and blind
  deconvolution.
\newblock \emph{Foundations of Computational Mathematics}, 2019.

\bibitem[Ma et~al.(2018)Ma, Olshevsky, Szepesvari, and
  Saligrama]{ma2018gradient}
Ma, Y., Olshevsky, A., Szepesvari, C., and Saligrama, V.
\newblock Gradient descent for sparse rank-one matrix completion for
  crowd-sourced aggregation of sparsely interacting workers.
\newblock In \emph{International Conference on Machine Learning}, pp.\
  3335--3344. PMLR, 2018.

\bibitem[Pimentel-Alarc{\'o}n et~al.(2016)Pimentel-Alarc{\'o}n, Boston, and
  Nowak]{pimentel2016characterization}
Pimentel-Alarc{\'o}n, D.~L., Boston, N., and Nowak, R.~D.
\newblock A characterization of deterministic sampling patterns for low-rank
  matrix completion.
\newblock \emph{IEEE Journal of Selected Topics in Signal Processing},
  10\penalty0 (4):\penalty0 623--636, 2016.

\bibitem[Recht et~al.(2010)Recht, Fazel, and Parrilo]{recht2010guaranteed}
Recht, B., Fazel, M., and Parrilo, P.~A.
\newblock Guaranteed minimum-rank solutions of linear matrix equations via
  nuclear norm minimization.
\newblock \emph{SIAM review}, 52\penalty0 (3):\penalty0 471--501, 2010.

\bibitem[Scieur et~al.(2017)Scieur, Roulet, Bach, and
  d'Aspremont]{scieur2017integration}
Scieur, D., Roulet, V., Bach, F., and d'Aspremont, A.
\newblock Integration methods and optimization algorithms.
\newblock \emph{Advances in Neural Information Processing Systems}, 30, 2017.

\bibitem[Shechtman et~al.(2015)Shechtman, Eldar, Cohen, Chapman, Miao, and
  Segev]{shechtman2015phase}
Shechtman, Y., Eldar, Y.~C., Cohen, O., Chapman, H.~N., Miao, J., and Segev, M.
\newblock Phase retrieval with application to optical imaging: a contemporary
  overview.
\newblock \emph{IEEE signal processing magazine}, 32\penalty0 (3):\penalty0
  87--109, 2015.

\bibitem[Tu et~al.(2016)Tu, Boczar, Simchowitz, Soltanolkotabi, and
  Recht]{tu2016low}
Tu, S., Boczar, R., Simchowitz, M., Soltanolkotabi, M., and Recht, B.
\newblock Low-rank solutions of linear matrix equations via procrustes flow.
\newblock In \emph{International Conference on Machine Learning}, pp.\
  964--973. PMLR, 2016.

\bibitem[Waldspurger \& Waters(2020)Waldspurger and
  Waters]{waldspurger2020rank}
Waldspurger, I. and Waters, A.
\newblock Rank optimality for the burer--monteiro factorization.
\newblock \emph{SIAM journal on Optimization}, 30\penalty0 (3):\penalty0
  2577--2602, 2020.

\bibitem[Zhang et~al.(2021)Zhang, Bi, and Lavaei]{zhang2021general}
Zhang, H., Bi, Y., and Lavaei, J.
\newblock General low-rank matrix optimization: Geometric analysis and sharper
  bounds.
\newblock \emph{arXiv preprint arXiv:2104.10356}, 2021.

\bibitem[Zhang et~al.(2019)Zhang, Sojoudi, and Lavaei]{zhang2019sharp}
Zhang, R.~Y., Sojoudi, S., and Lavaei, J.
\newblock Sharp restricted isometry bounds for the inexistence of spurious
  local minima in nonconvex matrix recovery.
\newblock \emph{J. Mach. Learn. Res.}, 20\penalty0 (114):\penalty0 1--34, 2019.

\end{thebibliography}

\appendix
\onecolumn

\section{PROOFS IN SECTION \ref{sec:exp-sol}}

\subsection{Proof of Proposition \ref{prop:rankrpolynomial}}
\begin{proof}
The condition $\mathcal{P}_{\mathbf{M^*}, \Omega, n, r} \in \mathcal{L}(\mathcal{G}, n, r)$ implies that $\mathcal{G}_1 = (\mathcal{V}, \mathcal{E}_1)$ is connected and non-bipartite. Since the graph is non-bipartite, there exists a cycle with an odd number of vertices $\mathcal{C}_{odd} = (\mathcal{V}_{odd}, \mathcal{E}_{odd})$ in $\mathcal{G}_1$ in which $\mathcal{E}_{odd} \subset \mathcal{E}_1$. To numerically find an odd cycle, the breadth first search method requires $\mathcal{O}(|\mathcal{V}| + |\mathcal{E}_1|) = \mathcal{O}(m^2)$ operations.
Without loss of generality, we assume that the set of vertices of the cycle is $\mathcal{V}_{odd}  = \{ 1,2,\dots, 2k+1\}$ and the set of edges is $ \mathcal{E}_{odd} = \{ (1,2), (2,3), \dots, (2k+1,1) \}$, where $k$ is a nonnegative integer. Suppose that the matrix ${\bf X}^*\in\mathbb{R}^{n\times r}$ satisfies ${\bf M}^*={\bf X}^*({\bf X}^*)^T$. We denote the $i$-th $r\times r$ block of ${\bf X}^*$ as ${\bf X}_i^*$ for all $i\in[m]$, i.e.,
\[ \mathbf{M^*}  = \mathbf{X}^*(\mathbf{X}^*)^T = \begin{bmatrix} \mathbf{X}^*_1 \\ \vdots\\  \mathbf{X}^*_m \end{bmatrix}\begin{bmatrix} (\mathbf{X}^*_1)^T & \dots & (\mathbf{X}^*_m)^T \end{bmatrix}. \]
Since $\mathcal{P}_{\mathbf{M^*}, \Omega, n, r} \in \mathcal{L}(\mathcal{G}, n, r)$, the block $\mathbf{X}^*_i$ is nonsingular for every $i\in[m]$, which further implies that the block ${\bf M}_{i,j}^*$ is nonsingular for all $i,j\in[m]$. Using the relation that ${\bf M}_{i,j}^* = {\bf X}_i^*({\bf X}_j^*)^T$, we can calculate that
\begin{align*}
    \left[  \prod_{i=1}^k \left({\bf M}_{2i-1,2i}^*\left({\bf M}_{2i,2i+1}^*\right)^{-T}\right)\right] {\bf M}_{2k+1,1}^* = {\bf X}_1^*({\bf X}_1^*)^T,
\end{align*}
Since the left-hand side only contains observed blocks, the matrix ${\bf X}_1^*({\bf X}_1^*)^T$ can be computed via observed blocks. Since computing the inverse of an $r\times r$ matrix and computing the product of two $r\times r$ matrices both require $\mathcal{O}(r^3)$ operations, the total number of operations required for computing ${\bf X}_1^*({\bf X}_1^*)^T$ is $\mathcal{O}[ (2k+1) r^3 ] = \mathcal{O}( m r^3 )$. In addition, computing the Cholesky decomposition of ${\bf X}_1^*({\bf X}_1^*)^T$ requires $\mathcal{O}(r^3)$ operations, which produces a matrix ${\bf X}_1^*{\bf R}$ for some orthogonal matrix ${\bf R}\in\mathbb{R}^{r\times r}$. 

With the knowledge of ${\bf X}_1^*{\bf R}$, we can recursively compute the block ${\bf X}_i^*$ using the connectivity of $\mathcal{G}_1$. More specifically, we use $\mathcal{P}\subset[m]$ to denote the set of vertices $i$ or which we have computed ${\bf X}_i^*{\bf R}$. We start with $\mathcal{P}=\{1\}$. At each iteration, we choose indices $i\in\mathcal{P}$ and $j\notin \mathcal{P}$ such that $(i,j)\in\mathcal{E}_1$. Such a pair of indices always exists unless $\mathcal{P}=[m]$, since the graph $\mathcal{G}_1$ is connected. Then, using the observation 
\[ {\bf M}_{j,i}^* = {\bf X}_j^*({\bf X}_i^*)^T = ({\bf X}_j^*{\bf R})({\bf X}_i^*{\bf R})^T , \]
we first compute the matrix ${\bf X}_j^*{\bf R}$ with $\mathcal{O}(r^3)$ operations and then add $j$ to the set $\mathcal{P}$. We stop the iteration when $\mathcal{P}=[m]$. After this process, we can concatenate ${\bf X}_i^*{\bf R}$ for all $i\in[m]$ to obtain the matrix ${\bf X}^*{\bf R}$. The number of iterations is $m-1$ and thus the total number of operations is $\mathcal{O}(mr^3)$. 

Summarizing the two parts, the total number of operations to compute the matrix ${\bf X}^*{\bf U}$ is $\mathcal{O}(m^2 + mr^3) = \mathcal{O}[n^2/r^2 + nr^2]$.

\end{proof}

\subsection{Gradient and Hessian of the Problem \eqref{eqn:obj-linear}}

Before proceeding with the analysis for the proofs in the remainder of this paper, we first derive the gradient and the Hessian of the objective function of the problem \eqref{eqn:obj-linear}. We omit the proof since the calculation can be done via basic calculus. The gradient of the objective function can be written as 
\begin{equation}\label{eqn:gradient}
    \nabla f({{\bf X}}) = 2({{\bf X}}{{\bf X}}^T - \mathbf{M}^*)_\Omega {{\bf X}}.
\end{equation}
Similarly, the quadratic variant of the Hessian can be written as
\begin{align}\label{eqn:Hessian}
    \Delta : \nabla^2 f({{\bf X}}) : \Delta &= 4 \langle ({{\bf X}}{{\bf X}}^T - \mathbf{M}^*)_\Omega, \Delta\Delta^T  \rangle + 2\left\| \left( {{\bf X}}\Delta^T + \Delta {{\bf X}}^T \right)_\Omega \right\|_F^2.
\end{align}

\subsection{Proof of Theorem \ref{thm:spuriousmany}}

\begin{proof}
Let $\mathcal{S}(\mathcal{G}_1)$ be a maximal independent set of $\mathcal{G}_1$ such that every vertex in the set has a self-loop. We define the global solution as ${\bf M}^* := {\bf x}^*({\bf x}^*)^T$, where 
\begin{align*}
    x^*_i := c_i,\quad \forall i \in \mathcal{S},\quad x^*_i := 0,\quad \forall i \not \in \mathcal{S}
\end{align*}
and $\{ c_i~|~ \forall i\in\mathcal{S} \}$ is a set of nonzero constants. We note that in the case when $r=1$, the factor ${\bf X}\in\mathbb{R}^n$ is a vector. Therefore, we represent it using the notation for vectors, i.e., ${\bf x}$. Considering the problem instance $\mathcal{P}_{{\bf M}^*, \Omega(\mathcal{G}), n, 1}$, the set of global solutions of the problem \eqref{eqn:obj-linear} is given by
\begin{align*}
    \mathcal{X}^* := \left\{ {\bf x}\in\mathbb{R}^n ~|~ x_i^2=c_i^2,~\forall i\in\mathcal{S},~ x_i=0,~\forall i\notin \mathcal{S}  \right\},
\end{align*}
which has the cardinality $|\mathcal{X}^*| = 2^{|\mathcal{S}|}$. 
%
%
%
For every global solution $\hat{{\bf x}} \in \mathcal{X}^*$, we have $\hat{{\bf x}}\hat{{\bf x}}^T = \mathbf{M}^*$. Thus, we know that $\hat{{\bf x}}$ is a first-order critical point of the problem \eqref{eqn:obj-linear}. For every $\Delta\in\mathbb{R}^n$, the quadratic form of the Hessian \eqref{eqn:Hessian} can be written as 
\begin{align*}
    \Delta : \nabla^2 f(\hat{{\bf x}}) : \Delta &= 2\left\| \left( \hat{{\bf x}}\Delta^T + \Delta \hat{{\bf x}}^T \right)_\Omega \right\|_F^2 = \sum_{i \in \mathcal{S}} 2(\Delta_i\hat{x}_i + \hat{x}_i\Delta_i)^2  +  \sum_{j \not \in \mathcal{S}} \sum_{\substack{i \in \mathcal{S}\\ (i,j) \in \mathcal{E}_1}} \left[2( \Delta_j\hat{x}_i)^2 +2( \hat{x}_i\Delta_j)^2 \right] \\
    & = \sum_{i \in \mathcal{S}} 4( \Delta_i \hat{x}_i)^2  + \sum_{j \not \in \mathcal{S}} \sum_{\substack{i \in \mathcal{S}\\ (i,j) \in \mathcal{E}_1}}  4 ( \Delta_j\hat{x}_i)^2 = \sum_{i \in \mathcal{S}} 4\hat{x}_i^2 \cdot \Delta_i^2  + \sum_{j \not \in \mathcal{S}} \left({\sum}_{\substack{i \in \mathcal{S}\\ (i,j) \in \mathcal{E}_1}}  4 \hat{x}_i^2\right) \cdot \Delta_j^2.
\end{align*}
The first term in the above expression corresponds to self-loops in $\mathcal{S}$, while the second term corresponds to the edges between $\mathcal{S}$ and $\mathcal{S}^c := \mathcal{V} \backslash \mathcal{S}$. We note that the edges whose endpoints are both in $\mathcal{S}^c$ do not contribute to the quadratic form. Since $\mathcal{S}$ is a maximal independent, we know that
\[ \{  i \in \mathcal{S}~|~ (i,j) \in \mathcal{E}_1 \} \neq \emptyset,\quad \forall j\notin \mathcal{S}. \]
As a result, it holds that
\[  \Delta : \nabla^2 f(\hat{{\bf x}}) : \Delta  > 0, \quad \forall \Delta \in \mathbb{R}^n\backslash\{0\}, \] 
which implies that the Hessian at the global solution $\hat{{\bf x}}$ is positive definite. Then, we perturb the global solution of the above problem to be 
\[ {\bf M}^*(\epsilon) := {\bf x}^*(\epsilon) \left[{\bf x}^*(\epsilon)\right]^T = ({\bf x}^* + \epsilon)({\bf x}^* + \epsilon)^T, \]
where ${\bf x}^*(\epsilon) := {\bf x}^* + \epsilon$ and $\epsilon \in \mathbb{R}^n$ is a small perturbation. We denote the problem \eqref{eqn:obj-linear} after perturbation as
     \begin{equation*}
      \min_{\mathbf{x}\in\mathbb{R}^n} \tilde{f}(\mathbf{x}; \epsilon),
     \end{equation*}
  where $\tilde{f}(\mathbf{x}; \epsilon) := \| (\mathbf{xx}^T - \mathbf{M}(\epsilon)^*)_{\Omega} \|_F^2$. For a generic perturbation $\epsilon$, all components of $\epsilon$ are nonzero and the problem $\mathcal{P}_{{\bf M}^*(\epsilon), \Omega,n,1}$ belongs to the class $\mathcal{L}(\mathcal{G}, n, 1)$. This implies that the global solution of the problem \eqref{eqn:prob2} is unique up to a sign flip.

The earlier argument implies that $\nabla_x \tilde{f}(\hat{{\bf x}}; 0) = 0 $ and $\nabla_{xx}\tilde{f}(\hat{{\bf x}}; 0) \succ 0$. Now, we analyze the relationship between the local minima of the original problem and those of the perturbed problem. We consider the equation $\nabla_{\bf x} \tilde{f}({\bf x};\epsilon) = 0$ near an unperturbed global minimum $\hat{{\bf x}}\in\mathcal{X}^*$. Since $(\hat{{\bf x}}; 0)$ is a solution to the gradient equation and the Jacobian matrix with respect to ${\bf x}$ is equal to the positive definite Hessian $\nabla^2 f(\hat{{\bf x}})$, the Implicit Function Theorem (IFT) states that there exists a unique solution $\hat{{\bf x}}(\epsilon)$ in a neighbourhood of $\hat{{\bf x}}$ for every $\epsilon$ with a small norm. In addition, the continuity of the Hessian implies that $\nabla_{\bf xx}\tilde{f}(\hat{\mathbf{x}}(\epsilon); \epsilon) \succ 0$. Otherwise, we can find a sequence of $\epsilon^k \rightarrow 0$ and another sequence $\{ y^k\}$ with $\| y_k \| = 1$ such that $(y^{k})^T\nabla_{\bf xx}\tilde{f}(\hat{\mathbf{x}}(\epsilon); \epsilon)y^k < 0$. By taking the limit along a convergent subsequence of $\{y^k\}$, we arrive at a contradiction. Therefore, every point in $\hat{\mathbf{x}}^*(\epsilon)$ still satisfies the second-order sufficient conditions. Since there are $2^{|\mathcal{S}|}$ global minima for the unperturbed problem, an analysis through the IFT implies that there are $2^{|\mathcal{S}|}$ strict local minima for the problem \eqref{eqn:prob2}. Hence, there are $2^{|\mathcal{S}|}-2$ spurious local minima for the perturbed problem \eqref{eqn:prob2}.


\end{proof}

\subsection{Proof of Corollary \ref{cor:singlemissing}}

\begin{proof}
We first prove the claim about the largest lower bound. Theorem \ref{thm:spuriousmany} implies that there are at least $2^{|\mathcal{S}(\mathcal{G}_1)|}-2$ spurious local minima for a problem instance in $ \mathcal{L}(\mathcal{G}, n, 1)$. For an arbitrary connected and non-bipartite graph $\mathcal{G}_1$ with $n$ vertices, the maximal possible size of a maximal independent set with self-loops is $n-1$. More specifically, the graph $\mathcal{G}^*_1$ that attains this maximal value is the star graph $K_{1,n-1}$ complemented with self-loops for the $n-1$ independent vertices. Hence, Theorem \ref{thm:spuriousmany} imples that there exists a problem instance in $\mathcal{L}(\mathcal{G}^*, n, 1)$ with at least $2^{n-1}-2$ local minima, where we define $\mathcal{G}^* := (\mathcal{G}_1^*, \emptyset)$.
%

We then consider the second claim of this corollary. Consider the measurement set $\Omega$ that observes all entries of the ground truth ${\bf M}^*$ except ${\bf M}_{12}^*$ and ${\bf M}_{21}^*$. In this case, we have $|\Omega| = n^2 - 2$. Choose the set of vertices to be $\mathcal{V}:=[n]$ and the set of edges to be $\mathcal{E}_1 := [n]\times [n] \backslash \{(1,2),(2,1)\}$. Then, the graph $\mathcal{G}:= (\mathcal{V},\mathcal{E}_1,\emptyset)$ satisfies that $\Omega = \Omega( \mathcal{G})$ and
the maximal independent set is $\mathcal{S} := \{ i, j \}$. Therefore, Theorem \ref{thm:spuriousmany} implies that there exists a problem instance in $\mathcal{L}(\mathcal{G}, n, 1)$ that has at least $2^{|\mathcal{S}|}-2 = 2$ spurious local minima.
\end{proof}

\subsection{Proof of Corollary \ref{cor:function}}

\begin{proof}
For a generic vector ${\bf x}^* \in \mathbb{R}^n$, all elements of ${\bf x}^*$ are nonzero and we can decompose ${\bf x}^*$ into ${\bf x}^0 + {\bf x}^1$, where
\begin{equation*}
    \begin{aligned}
        x_i^0 &:= x_i^*, && \forall i\in \mathcal{S}, \quad \quad x_i^0 := 0, &&&& \forall i\notin \mathcal{S},\\
        x_i^1 &:= 0, && \forall i\in \mathcal{S}, \quad \quad x_i^1 := x_i^*, &&&& \forall i\notin \mathcal{S}. 
    \end{aligned}
\end{equation*}
We first consider the problem $\mathcal{P}_{{\bf x}^0({\bf x}^0)^T, \Omega(\mathcal{G}), n, 1}$. Using a similar proof as Theorem \ref{thm:spuriousmany}, we can prove that this problem has $2^{|\mathcal{S}|}$ equivalent global solutions in formulation \eqref{eqn:obj-linear}, which are described by the set
\[ \mathcal{X}^* := \left\{ {\bf x}\in\mathbb{R}^n ~|~ x_i^2 = (x_i^*)^2,~\forall i\in\mathcal{S},~x_i=0,~\forall i\notin \mathcal{S} \right\}. \]
In addition, the Hessian is positive definite at every global solution $\hat{\bf x}\in\mathcal{X}^*$. Using a similar argument as the proof of Theorem \ref{thm:spuriousmany}, we conclude that there exists a small constant $r_{{\bf x}^0}>0$ such that the conditions
\begin{align}
\label{eqn:pf2-1} 
\|\epsilon\| \leq r_{{\bf x}^0},\quad \epsilon_i \neq 0,\quad \forall i\in [m] 
\end{align}
imply that the problem $\mathcal{P}_{{\bf x}^0(\epsilon)[{\bf x}^0(\epsilon)]^T, \Omega(\mathcal{G}), n, 1}$ has at least $2^{|\mathcal{S}|}-2$ spurious local minima in formulation \eqref{eqn:obj-linear}, where ${\bf x}^0(\epsilon) := {\bf x}^0+\epsilon$. Moreover, the MC problem is ``scale-free'' in the sense that the formulation \eqref{eqn:obj-linear} of the problem $\mathcal{P}_{{\bf x}'({\bf x}')^T, \Omega(\mathcal{G}), n, 1}$ has spurious local minima if and only if that of the problem $\mathcal{P}_{(c{\bf x}')(c{\bf x}')^T, \Omega(\mathcal{G}), n, 1}$ has spurious local minima, where ${\bf x}'\in\mathbb{R}^n$ is an arbitrary vector and $c\neq 0$ is a constant. Therefore, we have the relation
\begin{align}
    \label{eqn:pf2-3}
    r_{c{\bf x}^0} = c \cdot r_{{\bf x}^0},\quad \forall c\neq 0.
\end{align}
Hence, it suffices to consider vectors ${\bf x}^0\in\mathcal{X}_0$, where
\[ \mathcal{X}_0 := \left\{ {\bf x}\in\mathbb{R}^n ~|~ x_i \neq 0,~ \forall i\in \mathcal{S},~ x_i = 0,~ \forall i\notin \mathcal{S},~ \|{\bf x}\| = 1\right\}. \]

Now, we consider a vector $\hat{\bf x}^0\in\mathbb{R}^n$ satisfying 
\[ \|\hat{\bf x}^0 - {\bf x}^0\| < r_{{\bf x}^0} / 2. \]
Then, as long as the generic perturbation $\epsilon$ satisfies $\|\epsilon\| \leq r_{{\bf x}^0} / 2$, the condition \eqref{eqn:pf2-1} implies that the problem $\mathcal{P}_{\hat{\bf x}^*(\epsilon)[\hat{\bf x}^*(\epsilon)]^T, \Omega(\mathcal{G}), n, 1}$ has at least $2^{|\mathcal{S}|}-2$ spurious local minima in formulation \eqref{eqn:obj-linear}, where $\hat{\bf x}^*(\epsilon) := \hat{\bf x}^0 + \epsilon$. This verifies the existence of a function ${h}_\mathcal{S}$ on the open set
\[ \mathcal{N}({\bf x}^0) := \left\{ {\bf x}\in\mathbb{R}^n ~|~ \|{\bf x} - {\bf x}^0\| < r_{{\bf x}^0} / 2 \right\},\quad \forall {\bf x}^0 \in\mathcal{X}_0. \]
Now, we consider the compact set
\[ \mathcal{X}_k := \left\{ {\bf x}\in\mathbb{R}^n ~|~ \min_{i\in\mathcal{S}}|x_i| \geq 1/k, ~ x_i = 0,~ \forall i\notin \mathcal{S},~ \|{\bf x}\| = 1\right\}, \quad k=1,2,\dots \]
The open set family $\{\mathcal{N}({\bf x}^0)~|~{\bf x}^0 \in \mathcal{X}_0\}$ comprises an open covering of the compact set $\mathcal{X}_k$. Hence, there exists a finite covering of the compact set $\mathcal{X}_k$ and thus there exists a small constant $\epsilon_k$ such that
\[ h_{\mathcal{S}}({\bf x}^0) \geq \epsilon_k,\quad \forall {\bf x}^0 \in \mathcal{X}_k. \]
Moreover, using the relation 
\[ \mathcal{X}_{k-1} \subset \mathcal{X}_{k}, \quad\forall k\geq 2, \quad \mathcal{X}_0 = \bigcup_{k=1}^\infty \mathcal{X}_k, \]
we can decrease the value of the function $h_{\mathcal{S}}$ to be
\[ h({\bf x}^0) := \epsilon_k, \quad \forall {\bf x}^0 \in \mathcal{X}_{k} \backslash \mathcal{X}_{k-1},~k\geq 2. \]
Using this definition, $h_\mathcal{S}$ reduces to $\min_{i\in\mathcal{S}}|x_i|$ and, with a little abuse of notations, we still write the new function as $h_\mathcal{S}( \min_{i\in\mathcal{S}}|x_i| )$.

Now, we can view the problem $\mathcal{P}_{{\bf x}^*({\bf x}^*)^T, \Omega(\mathcal{G}), n, 1}$ as the perturbed problem, where the generic perturbation is given by ${\bf x}^1$. The above analysis implies that the following condition holds
\begin{align*}
    \| {\bf x}_{{\mathcal{S}^c}^*}\| = \| {\bf x}^1\| \leq h_{\mathcal{S}}\left(\min_{i\in\mathcal{S}}|x_i|\right) \cdot \|{\bf x}^0\| = h_{\mathcal{S}}\left(\min_{i\in\mathcal{S}}|x_i|\right) \cdot \|{\bf x}_\mathcal{S}^*\|
\end{align*}
which demonstrates the existence of at least $2^{|\mathcal{S}|}-2$ spurious local minima in formulation \eqref{eqn:obj-linear}.

\end{proof}

\subsection{Proof of Lemma \ref{lem:rankr-1}}
\begin{proof}
We follow a similar proof construction as in Theorem \ref{thm:spuriousmany}. Let $\mathcal{D}$ be the set of full-rank diagonal matrices and $\mathcal{D}_{\{1,-1\}}$ be the set of diagonal matrices with the diagonal entries being $+1$ or $-1$. Suppose that $\mathcal{S}$ is a maximal independent set of $\mathcal{G}_1$ in which every node has a self-loop. Then, we define the ground truth matrix $\mathbf{M^*} := \mathbf{X^*(X^*)}^T$, where $\mathbf{X}^* \in \mathbb{R}^{n\times r}$ satisfies
\[ \mathbf{X}^*_i = 0,\quad \forall i \not \in \mathcal{S},\quad  \mathbf{X}^*_i \in \mathcal{D},\quad \forall i \in \mathcal{S}, \]
where ${\bf X}_i$ is the $i$-th block of ${\bf X}\in\mathbb{R}^{n\times r}$; see the definition in the proof of Proposition \ref{prop:rankrpolynomial}.
Then, the set of global solutions is given by
\begin{align*} 
    \mathcal{X}^* := \Big\{ {\bf X}\in\mathbb{R}^{n\times r} ~|~  \mathbf{X}_i = \mathbf{X}^*_i\mathbf{D},~ \forall i \in \mathcal{S},~ \mathbf{D} \in \mathcal{D}_{\{1,-1\}}, \quad \mathbf{X}_i = 0,~ \forall i \not \in \mathcal{S} \Big\},
\end{align*}
For every global solution $\hat{{\bf X}} \in \mathcal{X}^*$, we have $\hat{{\bf X}}\hat{{\bf X}}^T = \mathbf{M}^*$. Thus, every global solution is a first-order critical point of the problem \eqref{eqn:obj-linear}. In addition, let $\Delta \in \mathbb{R}^{n \times r}$ be an arbitrary direction matrix with its $r \times r$ block matrices denoted as $\Delta_1, \Delta_2,\dots,\Delta_m$. Then, the quadratic variant of the Hessian \eqref{eqn:Hessian} in the direction $\Delta$ can be written as
\begin{align}
\Delta : \nabla^2 f(\hat{{\bf X}}) : \Delta &= 2\left\| \left( \hat{{\bf X}}\Delta^T + \Delta \hat{{\bf X}}^T \right)_\Omega \right\|_F^2 \label{eqn:pf2-4}
\\
& = \sum_{i \in \mathcal{S}} \|\Delta_i\hat{\bf X}_i^T + \hat{\bf X}_i\Delta_i^T\|_F^2  + 2 \sum_{j \not \in \mathcal{S}}  \left( \sum_{\substack{i \in \mathcal{S}\\ (i,j) \in \mathcal{E}_1}} \| \Delta_j\hat{\bf X}_i^T\|_F^2 + \sum_{\substack{i \in \mathcal{S}\\ (i,j) \in \mathcal{E}_2}} \| (\Delta_j\hat{\bf X}_i^T)_{nd}\|_F^2 \right)
\nonumber
\\
& + 2 \sum_{j \in \mathcal{S}} \sum_{\substack{i \in \mathcal{S} \\ (i,j) \in \mathcal{E}_2}}  \|(\Delta_j\hat{\bf X}_i^T + \hat{\bf X}_j\Delta_i^T)_{nd}\|_F^2  
\nonumber
\end{align}
where $(\cdot)_{nd}$ is the projection onto the matrix space with a zero diagonal. We note that the first term in \eqref{eqn:pf2-4} corresponds to self-loops in $\mathcal{G}_1$, while the second term corresponds to edges between $\mathcal{S}$ and $\mathcal{S}^c$. The edges whose endpoints are both in $\mathcal{S}^c$ do not contribute to the quadratic form. Moreover, the last term corresponds to partial observations with nondiagonal entries within the independent set $\mathcal{S}$.

Now, we aim to prove that the Hessian at $\hat{\bf X}$ is positive definite in the tangent space of $\mathcal{W}^{n\times r}$, namely,
\[ \Delta : \nabla^2 f(\hat{{\bf X}}) : \Delta > 0,\quad \forall \Delta\in\mathbb{R}^{n\times r}\backslash \{0\},\quad \Delta_1 \text{ is lower triangular}. \]
We assume that $\Delta : \nabla^2 f(\hat{{\bf X}}) : \Delta=0$ for some $\Delta\in\mathbb{R}^{n\times r}$ such that $\Delta_1$ is lower triangular. Under this assumption, all three terms in \eqref{eqn:pf2-4} are equal to zero. Considering the second term, since $\hat{\bf X}_i$ is full-rank, we have 
\[ \Delta_j = 0,\quad \forall j\notin\mathcal{S}. \]
For the first term, $\|\Delta_i\hat{\bf X}_i^T + \hat{\bf X}_i\Delta_i^T\|_F^2 $ is zero only if $\Delta_i\hat{\bf X}_i^T = - \hat{\bf X}_i\Delta_i^T$, i.e., $\Delta_i\hat{\bf X}_i^T$ is skew-symmetric. Since $\hat{\bf X}_i$ is a diagonal matrix with nonzero diagonal entries, the diagonal entries of $\Delta_i$ must be zero for all $i \in \mathcal{S}$. Without loss of generality, we assumed that vertex $1 \in \mathcal{S}$. This is because we can equivalently fix the block ${\bf X}_i$ to be lower-diagonal for any $i\in\mathcal{S}$ and consider a similarly constrained optimization problem. 
Then, since $\Delta_1$ must be lower triangular, we have
\[ \Delta_1 = 0, \]
We define the set
\[ \mathcal{S}_0 := \left\{ i\in\mathcal{S}~|~ \Delta_i = 0 \right\}. \]
We have shown that $i\in\mathcal{S}_0$ and aim to prove that $\mathcal{S}_0 = \mathcal{S}$. Since the induced subgraph $\mathcal{G}_2[\mathcal{S}]$ is connected, there exists a vertex $j \in \mathcal{S}$ such that $(1, j) \in \mathcal{E}_2$. Considering the third term in \eqref{eqn:pf2-4},
we have 
\[ (\Delta_j\hat{\bf X}_1^T + \hat{\bf X}_j\Delta_1^T)_{nd} = (\Delta_j\hat{\bf X}_1^T)_{nd} = 0, \]
which implies that $\Delta_j = 0$ because $\hat{\bf X}_j$ is a diagonal matrix with nonzero diagonal entries. Hence, we have proved that $j\in\mathcal{S}_0$.
%
By the connectivity of $\mathcal{G}_2[\mathcal{S}]$, we can inductively prove that all elements in $\mathcal{S}$ belong to $\mathcal{S}_0$. Therefore, it holds that $\Delta =0$ and the quadratic form of Hessian is zero only when $\Delta=0$. As a result, the Hessian is positive definite at every global solution $\hat{{\bf X}} \in \mathcal{X}^*$ of the problem \eqref{eqn:probres}.

Then, we perturb the ground truth of the above problem instance to be 
\[ {\bf M}^*(\epsilon) := {\bf X}^*(\epsilon) \left[{\bf X}^*(\epsilon)\right]^T = ({\bf X}^* + \epsilon)({\bf X}^* + \epsilon)^T, \]
where ${\bf X}^*(\epsilon) := {\bf X}^* + \epsilon$ and $\epsilon \in \mathbb{R}^{n \times r }$ is a small perturbation. Similar to Theorem \ref{thm:spuriousmany}, for a generic perturbation $\epsilon$, all block components of $\epsilon$ are nonzero and full-rank. Therefore, the problem $\mathcal{P}_{{\bf M}^*(\epsilon), \Omega,n,r}$ belongs to the class $\mathcal{L}(\mathcal{G}, n, r)$. This implies that the global solution of the problem \eqref{eqn:probres} is unique up to a right-multiplication with $\mathbf{D} \in \mathcal{D}_{\{1,-1\}}$. Since there are $2^{r|\mathcal{S}|}$ global minima for the unperturbed problem, IFT implies that there are $2^{r|\mathcal{S}|}$ strict local minima for the perturbed problem. Hence, there are $2^{r|\mathcal{S}|}-2^r$ spurious local minima for the perturbed problem.
\end{proof}

\subsection{Proof of Corollary \ref{cor:rsinglemissing}}

\begin{proof}

We first consider the largest lower bound on the number of spurious local minima. Similar to Corollary \ref{cor:singlemissing}, since an instance in $\mathcal{L}(\mathcal{G}, n, r)$ can have a maximum independent set of size at most $|\mathcal{S}(\mathcal{G}_1)| = m-1 = {n}/{r} - 1 $, Theorem \ref{thm:rspuriousmany} implies that the largest lower bound on the number of spurious local solution classes is $2^{r({n}/{r}-2)}-1 = 2^{n -2r} -1$. 

We prove the second part of this corollary next. We choose $\mathcal{G}_1$ to be a graph with $m$ vertices and $\genfrac(){0pt}{}{m}{2} - 1 $ edges, where $(i,j)$ is the only missing edge. Then, the maximal independent set is $\mathcal{S} := \{ i, j \}$. Since $\mathcal{G}_2[\mathcal{S}]$ must be connected, the nondiagonal entries of the block $\mathbf{M}^*_{i,j}$ are observed. Thus, only $2r$ entries are not observed and $|\Omega|=n^2 - 2r$. Furthermore, Theorem \ref{thm:rspuriousmany} implies that there exists a problem instance in $\mathcal{L}(\mathcal{G},n,r)$ with at least $2^{r(|\mathcal{S}|-1)}-1 = 2^r-1$ equivalent classes of spurious local minima.
\end{proof}

\subsection{Proof of Theorem \ref{thm:characterization}}

\begin{proof}
The proof is similar to those of Theorems \ref{thm:spuriousmany} and \ref{thm:rspuriousmany}. We can prove that the Hessian is positive definite at all global solutions for each loss function $g(\cdot)$ satisfying Assumption \ref{asp:general}.
\end{proof}

\section{PROOFS IN SECTION \ref{sec:more-obs}}

\subsection{Proof of Proposition \ref{prop:nospurious}}

\begin{proof}
For every matrix ${\mathbf{X}} \in \mathbb{R}^{n\times r}$, we denote ${\mathbf{X}}_i$ as the $i$-th $r\times r$ block of ${\bf X}$ for all $i\in[m]$. Because each block $\mathbf{M}^*_{i,j}$ is assumed to be full rank, the block $\mathbf{X}^*_i$ is also full rank for all $i\in[m]$, where ${\bf X}^*$ satisfies ${\bf M}^*={\bf X}^*({\bf X}^*)^T$. It is desirable show that every first-order critical point is either a global solution or a saddle point with a strict descent direction. For every $i\in[m]$, the gradient of the problem \eqref{eqn:obj-linear} with respect to the $i$-th block ${\bf X}_i$ is
\begin{equation*}
    \nabla_{\mathbf{X}_i}f(\mathbf{X})  = \begin{cases}
    2(\mathbf{X}_i\mathbf{X}_k^T - \mathbf{X}^*_i(\mathbf{X}^*_k)^T)\mathbf{X}_k , & \text{if }i \not = k\\
    \sum_{j=1}^{m}         2(\mathbf{X}_k\mathbf{X}_j^T - \mathbf{X}^*_k(\mathbf{X}^*_j)^T)\mathbf{X}_j  , & \text{if }i = k.
    \end{cases}
\end{equation*}
Let $\hat{\bf X}\in\mathbb{R}^{n\times r}$ be a first-order critical point of problem \eqref{eqn:obj-linear}. 

We first consider the case when $\hat{\mathbf{X}}_k$ is nonsingular. For every $i\in[m]\backslash\{k\}$, the condition $\nabla_{\mathbf{X}_i}f(\mathbf{X})=0$ implies that
\[ \hat{\mathbf{X}}_i\hat{\mathbf{X}}_k^T - \mathbf{X}^*_i(\mathbf{X}^*_k)^T = 0. \]
Substituting the above equations into the equation $\nabla_{\mathbf{X}_k}f(\mathbf{X}) = 0$, we obtain
\begin{align*}
    (\mathbf{X}_k\mathbf{X}_k^T - \mathbf{X}^*_k(\mathbf{X}^*_k)^T)\mathbf{X}_k = 0,
\end{align*}
which implies that $\hat{\mathbf{X}}_k\hat{\mathbf{X}}_k^T = \mathbf{X}^*_k(\mathbf{X}^*_k)^T$. Therefore, the matrix $\hat{\bf X}$ is a global solution of the problem \eqref{eqn:obj-linear} in this case.

Now, we consider the case when $\hat{\mathbf{X}}_k$ is singular. We choose a vector ${\bf y}_k\in\mathbb{R}^n$ such that
\[ \hat{\bf X}_k{\bf y}_k = 0,\quad \|{\bf y}_k\| = 1. \]
%
%
Given a small constant $\epsilon>0$, the $i$-th block direction $\Delta\in\mathbb{R}^{n\times r}$ is defined as
\begin{align*}
    \Delta_i := \begin{cases} {\bf z}_i {\bf y}_k^T, & \text{if } i\neq k\\
    \epsilon\mathbf{y}_k\mathbf{y}_k^T, & \text{if } i = k,\end{cases}
\end{align*}
where ${\bf z}_i\in\mathbb{R}^{n}$ is arbitrary. The above definition directly implies that $\hat{\mathbf{X}}_k\Delta_i^T = 0$ for all $i\in[m]$. Then, we obtain
\begin{equation*}
    \begin{aligned}
        4\left\langle (\hat{\mathbf{X}}\hat{\mathbf{X}}^T - \mathbf{X}^*({\bf X}^*)^T)_\Omega , \Delta\Delta^T \right\rangle  &= - 4\textbf{tr}\left[\mathbf{X}^*_k(\mathbf{X}^*_k)^T\Delta_k\Delta_k^T\right] - \sum_{j=1, j \not = i}^{m} 8 \textbf{tr}\left(\mathbf{X}^*_j(\mathbf{X}^*_k)^T\Delta_k\Delta_j^T\right) \\
        &= \sum_{j=1, j \not = i}^{m} -8\textbf{tr}\left[\mathbf{X}^*_j(\mathbf{X}^*_k)^T{\bf y}_k{\bf z}_j^T\right] \cdot \epsilon + \mathcal{O}(\epsilon^2),
    \end{aligned}
\end{equation*}
and
\begin{equation*}
    2\|( {{\bf X}}\Delta^T + \Delta {{\bf X}}^T )_\Omega\|_F^2 = 8\|{\bf X}_k\Delta_k^T\|_F^2 + 4 \sum_{j=1,j\neq k}^m\|{\bf X}_j\Delta_k^T\|_F^2 =  \mathcal{O}(\epsilon^2)
\end{equation*}
Combining two estimates above, the quadratic form of the Hessian \eqref{eqn:Hessian} can be written as
\[ \Delta : \nabla^2 f(\hat{\mathbf{X}}) : \Delta =  \sum_{j=1, j \not = i}^{m} -8\textbf{tr}\left[\mathbf{X}^*_j(\mathbf{X}^*_k)^T{\bf y}_k{\bf z}_j^T\right] \cdot \epsilon + \mathcal{O}(\epsilon^2). \]
Since ${\bf X}^*_i$ is nonsingular for all $i\in[m]$, it holds that $\mathbf{X}^*_j(\mathbf{X}^*_k)^T{\bf y}_k \neq 0$. Choosing 
\[ {\bf z}_j := \mathbf{X}^*_j(\mathbf{X}^*_k)^T{\bf y}_k, \]
we obtain
\[ \Delta : \nabla^2 f(\hat{\mathbf{X}}) : \Delta =  \sum_{j=1, j \not = i}^{m} -8\|\mathbf{X}^*_j(\mathbf{X}^*_k)^T{\bf y}_k\|^2 \cdot \epsilon + \mathcal{O}(\epsilon^2). \]
Hence, the quadratic form of the Hessian is negative with a sufficiently small $\epsilon$ and $\hat{\bf X}$ is a strict saddle point. 

Combining the two cases, we conclude that every second-order critical point of the problem \eqref{eqn:obj-linear} is a global minimum. 

\end{proof}

\section{PROOFS IN SECTION \ref{sec:gd}}

\subsection{Proof of Lemma \ref{lem:gd-1}}


In this proof and the following proofs for Section \ref{sec:gd}, we consider the instance of the MC problem constructed in Section \ref{sec:exp-sol}. For completeness, we repeat the instance here. In the unperturbed case, the ground truth matrix is defined as ${\bf M}^* := {\bf x}^*({\bf x}^*)^T\in\mathbb{R}^{n\times n}$, where vector ${\bf x}^*\in\mathbb{R}^{n}$ satisfies
\[ {\bf x}^*_{2k-1} = 1,\quad \forall k=1,\dots, \lceil n/2 \rceil,\quad {\bf x}^*_{2k} = 0,\quad \forall k=1,\dots, \lfloor n / 2\rfloor. \]
The measurement set $\Omega$ is given by
\[ \Omega := \left\{ (j,j), (2k,j), (j,2k) ~|~ j=1,\dots,n,~ k=1,\dots, \lfloor n/2\rfloor  \right\}. \]
It has been proved in Section \ref{sec:exp-sol} that the problem \eqref{eqn:obj-general} has $2^{\lceil n/2 \rceil}$ global solutions, which are given by the set
\[ \mathcal{X}^* := \left\{ {\bf x}\in \mathbb{R}^n ~|~ {\bf x}_{2k} = 0,~k=1,\dots, \lfloor n/2\rfloor,~ {\bf x}_{2k+1}^2 = 1,~ k=1,\dots, \lceil n/2 \rceil \right\}. \]
For each vector ${\bf x}\in\mathbb{R}^n$, we denote 
\[ {\bf x}^{o} := (x_1,x_3,\dots,x_{\lceil n/2 \rceil}),\quad {\bf x}^{e} := (x_2,x_4,\dots,x_{\lfloor n/2 \rfloor}). \]
Before presenting the proof of Lemma \ref{lem:gd-1}, we first state three technical lemmas.
\begin{lemma}\label{lem:pf5-1}
Suppose that Assumption 1 holds with $(\delta, 1)$, and that $\hat{{\bf x}}$ is a first-order critical point of the problem \eqref{eqn:obj-general}. Then, it holds that
\[ \|\hat{{\bf x}}^e\|  \| \hat{{\bf x}} \| \leq 2\sqrt{2}\delta \| \left({\bf M} - {\bf M}^*\right)_\Omega\|_F, \]
where we define ${\bf M} := {\bf \hat{x}\hat{x}}^T$.
\end{lemma}
\begin{proof}

Utilizing the first-order optimality condition and the gradient in \eqref{eqn:gradient}, it holds that
\begin{align*}
    \left \langle \nabla g\left[ \left( \hat{{\bf x}}\hat{{\bf x}}^T - {\bf x}^*{\bf x}^* \right)_\Omega \right], \hat{{\bf x}}\Delta^T \right\rangle = \int_0^1 \left( {\bf M} - {\bf M}^* \right) : \nabla^2g\left[({{\bf M}^*})_\Omega+t\left( {\bf M} - {\bf M}^* \right)_\Omega\right] : \hat{{\bf x}}\Delta^T ~dt = 0, \quad \forall \Delta \in \mathbb{R}^n,
\end{align*}
where the first equality is from Taylor's expansion. For every fixed number $t\in[0,1]$, the proof of Theorem 1 in \cite{bi2020global} implies that
\begin{align*}
    &\left( {\bf M} - {\bf M}^* \right) : \nabla^2g\left[{{\bf M}^*}_\Omega+t\left( {\bf M} - {\bf M}^* \right)_\Omega\right] : \hat{{\bf x}}\Delta^T \\
    &\hspace{16em}\geq \left\langle \left( {\bf M} - {\bf M}^* \right)_\Omega, \left( \hat{{\bf x}}\Delta^T \right)_\Omega \right\rangle - 2\sqrt{2}\delta \| \left( {\bf M} - {\bf M}^* \right)_\Omega\|_F \| \left( \hat{{\bf x}}\Delta^T \right)_\Omega\|_F .
\end{align*}
Integrating over $t$, it follows that
\begin{align}\label{eqn:pf5-1}
    \left\langle \left( {\bf M} - {\bf M}^* \right)_\Omega, \left( \hat{{\bf x}}\Delta^T \right)_\Omega \right\rangle \leq 2\sqrt{2}\delta \| \left( {\bf M} - {\bf M}^* \right)_\Omega\|_F \| \left( \hat{{\bf x}}\Delta^T \right)_\Omega\|_F.
\end{align}
By choosing
\[ \Delta_{2k+1} = 0,\quad k=1,\dots, \lceil n/2 \rceil,\quad \Delta_{2k}=\hat{x}_{2k},\quad k=1,\dots, \lfloor n/2 \rfloor, \]
we obtain
\begin{align*}
    \left\langle \left( {\bf M} - {\bf M}^* \right)_\Omega, \left( \hat{{\bf x}}\Delta^T \right)_\Omega \right\rangle = \| \hat{{\bf x}}^e \|^2 \| \hat{{\bf x}}\|^2,\quad \|\left( \hat{{\bf x}}\Delta^T \right)_\Omega\|_F = \| \hat{{\bf x}}^e \| \| \hat{{\bf x}}\|.
\end{align*}
Substituting the above two equalities into \eqref{eqn:pf5-1}, we have
\[ \|\hat{{\bf x}}^e\|^2  \| \hat{{\bf x}} \|^2 \leq 2\sqrt{2}\delta \| \left({\bf M} - {\bf M}^*\right)_\Omega\|_F \|\hat{{\bf x}}^e\| \| \hat{{\bf x}} \|. \]
The above inequality implies that
\[ \|\hat{{\bf x}}^e\|  \| \hat{{\bf x}} \| \leq 2\sqrt{2}\delta \| \left({\bf M} - {\bf M}^*\right)_\Omega\|_F \quad \text{or}\quad \hat{{\bf x}}^e = 0, \]
since $\|\hat{{\bf x}}^e\| \| \hat{{\bf x}} \|=0$ if and only if $\hat{{\bf x}}^e=0$. In both cases, the claim of this lemma holds.
\end{proof}

\begin{lemma}\label{lem:pf5-3}
Let $\mathcal{D}$ be the set of $r \times r$ diagonal matrices and $\mathcal{D}_{\{1,-1\}}$ be set of $r \times r$ diagonal matrices with the diagonal entries $+1$ or $-1$. Under the same setting as Lemma \ref{lem:rankr-1}, consider the $n \times n$ ground truth matrix $\mathbf{M^*} := \mathbf{X}^*(\mathbf{X}^*)^T$ such that $n = mr $, where
\[ \mathbf{X}_i^* \in \mathcal{D},\quad \forall i\in\mathcal{S}(\mathcal{G}_1),\quad \mathbf{X}_i^* = \mathbf{0},\quad \forall i \notin \mathcal{S}(\mathcal{G}_1). \]
Let $\mathbf{D} \in \mathbb{R}^{n \times n}$ be an arbitrary matrix with its diagonal blocks denoted as $\mathbf{D}_1, \mathbf{D}_2, \dots, \mathbf{D}_{m}$ such that $\mathbf{D}_i \in \mathcal{D}_{\{1,-1\}}$ for all  $i \in [m]$. Then, the problem instance $\mathcal{P}_{{\bf M}^*, \Omega(\mathcal{G}), n, r}$ in formulation \eqref{eqn:obj-linear} satisfies that
\[ \nabla f(\mathbf{DX})=\mathbf{D}\nabla f(\mathbf{X}), \quad \forall \mathbf{X} \in \mathbb{R}^{n \times r}. \]
\end{lemma}

\begin{proof}
Since the graph $\mathcal{G}$ satisfies the conditions in Lemma \ref{lem:rankr-1}, the gradient $\nabla f(\mathbf{X})$ for problem \eqref{eqn:obj-linear} can be written as
\begin{equation*}
\nabla_{\mathbf{X}_k}f(\mathbf{DX})  = 2\begin{cases}
 (\mathbf{X}_k\mathbf{X}_k^T - \mathbf{X}_k^*(\mathbf{X}_k^*)^T)\mathbf{X}_k + \sum_{(k,j) \in \mathcal{E}_1} \mathbf{X}_k\mathbf{X}_j^T\mathbf{X}_j + \sum_{(k,j) \in \mathcal{E}_2} (\mathbf{X}_k\mathbf{X}_j^T - \mathbf{X}_k^*(\mathbf{X}_j^*)^T)_{nd}\mathbf{X}_j  , & \text{if }  k \in \mathcal{S}\\
\sum_{(k,j) \in \mathcal{E}_1} \mathbf{X}_k\mathbf{X}_j^T\mathbf{X}_j + \sum_{(k,j) \in \mathcal{E}_2} (\mathbf{X}_k\mathbf{X}_j^T)_{nd}\mathbf{X}_j  , & \text{if }k \not \in \mathcal{S},
\end{cases}
\end{equation*}
where $(\cdot)_{nd}$ is the projection onto the matrix space with the zero diagonals. Note that blocks of the transformed variable $\mathbf{DX}$ are $\mathbf{D}_i\mathbf{X}_i$ for all $i \in [m]$. First, we consider the change in the $i$-th block of the gradient function for $i \in \mathcal{S}$:
\begin{align}\label{eqn:pf5-7}
    \nabla_{\mathbf{X}_i}f(\mathbf{DX}) & =
 2  \Big( (\mathbf{D}_i\mathbf{X_i}\mathbf{X}_i^T\mathbf{D}_i^T - \mathbf{X}_i^*(\mathbf{X}_i^*)^T)\mathbf{D}_i\mathbf{X}_i + \sum_{(i,j) \in \mathcal{E}_1} \mathbf{D}_i\mathbf{X}_i\mathbf{X}_j^T\mathbf{D}_j^T\mathbf{D}_j\mathbf{X}_j + \\
 \nonumber& \sum_{(i,j) \in \mathcal{E}_2} (\mathbf{D}_i\mathbf{X}_i\mathbf{X}_j^T\mathbf{D_j}^T - \mathbf{X}_i^*(\mathbf{X}_j^*)^T)_{nd}\mathbf{D}_j\mathbf{X}_j  \Big).
\end{align}
Using $\mathbf{D}_i^T\mathbf{D}_i =\mathbf{I}$ and $\mathbf{X}^*_i \in \mathcal{D}$, the first term in \eqref{eqn:pf5-7} can be written as
\begin{align*}
    (\mathbf{D}_i\mathbf{X_i}\mathbf{X}_i^T\mathbf{D}_i^T - \mathbf{X}_i^*(\mathbf{X}_i^*)^T)\mathbf{D}_i\mathbf{X}_i & = \mathbf{D}_i\mathbf{X_i}\mathbf{X}_i^T\mathbf{X}_i - \mathbf{X}_i^*(\mathbf{X}_i^*)^T \mathbf{D}_i\mathbf{X}_i \\
    & = \mathbf{D}_i(\mathbf{X_i}\mathbf{X}_i^T - \mathbf{X}_i^*(\mathbf{X}_i^*)^T)\mathbf{X}_i,
\end{align*}
where the second equality is justified by the commutative property of diagonal matrix multiplication. Similarly, the second term in \eqref{eqn:pf5-7} can be written as
\[  \sum_{(i,j) \in \mathcal{E}_1} \mathbf{D}_i\mathbf{X}_i\mathbf{X}_j^T\mathbf{D}_j^T\mathbf{D}_j\mathbf{X}_j =  \mathbf{D}_i \sum_{(i,j) \in  \mathcal{E}_1} \mathbf{X}_i\mathbf{X}_j^T\mathbf{X}_j. \]
For the last term in \eqref{eqn:pf5-7}, we use the relation $\mathbf{X}^*_i \in \mathcal{D}$ and the fact that $(\cdot)_{nd}$ is nonzero only at positions associated with the nondiagonal entries to obtain
\begin{align*}
\sum_{(i,j) \in \mathcal{E}_2} (\mathbf{D}_i\mathbf{X}_i\mathbf{X}_j^T\mathbf{D_j}^T - \mathbf{X}_i^*(\mathbf{X}_j^*)^T)_{nd}\mathbf{D}_j\mathbf{X}_j & = \sum_{(i,j) \in \mathcal{E}_2} \mathbf{D}_i(\mathbf{X}_i\mathbf{X}_j^T - \mathbf{X}_i^*(\mathbf{X}_j^*)^T)_{nd}\mathbf{D}_j^T\mathbf{D}_j\mathbf{X}_j \\
& = \mathbf{D}_i \sum_{(i,j) \in \mathcal{E}_2} (\mathbf{X}_i\mathbf{X}_j^T - \mathbf{X}_i^*(\mathbf{X}_j^*)^T)_{nd}\mathbf{X}_j.
\end{align*}
Thus, we have
\[ \nabla_{\mathbf{X}_i}f(\mathbf{DX}) = \mathbf{D}_i\nabla_{\mathbf{X}_i}f(\mathbf{X}), \quad \forall i \in \mathcal{S}. \]
Now, we consider the change in the $i$-th block of the gradient function for $i \not \in \mathcal{S}$:
\begin{align*}
   \nabla_{\mathbf{X}_i}f(\mathbf{DX}) & =
2 \left( \sum_{(i,j) \in \mathcal{E}_1}      \mathbf{D}_i\mathbf{X}_i\mathbf{X}_j^T\mathbf{D}_j^T\mathbf{D}_j\mathbf{X}_j + \sum_{(i,j) \in \mathcal{E}_2} (\mathbf{D}_i\mathbf{X}_i\mathbf{X}_j^T\mathbf{D}_j^T)_{nd}\mathbf{D}_j\mathbf{X}_j  \right) \\
& = 2 \mathbf{D}_i\left( \sum_{(i,j) \in \mathcal{E}_1}      \mathbf{X}_i\mathbf{X}_j^T\mathbf{X}_j + \sum_{(i,j) \in \mathcal{E}_2} (\mathbf{X}_i\mathbf{X}_j^T)_{nd}\mathbf{X}_j  \right)  = \mathbf{D}_i\nabla_{\mathbf{X}_i}f(\mathbf{X}),
\end{align*}
where the second equality holds by a similar argument as in the case of $i\in\mathcal{S}$. 
%
Consequently, we have
\[ \nabla_{\mathbf{X}_i}f(\mathbf{DX}) = \mathbf{D}_i\nabla_{\mathbf{X}_i}f(\mathbf{X}), \quad \forall i \not \in \mathcal{S}. \]
Combining the two cases, it follows that
\[ \nabla f(\mathbf{DX}) = \mathbf{D}\nabla f(\mathbf{X}). \]

\end{proof}


\begin{lemma}\label{lem:pf5-2}
Consider the case $r=1$. Given an arbitrary point $x_0 \in \mathbb{R}^n$, let $\hat{x} \in \mathbb{R}^n$ denote a point with the property that the gradient flow \eqref{eqn:grad-flow} initialized at ${\bf x}_0$ converges to $\hat{\bf x}$. For every diagonal matrix ${\bf D}\in\mathbb{R}^{n\times n}$ that satisfies
\[ {\bf D}_{ii}^2 = 1,\quad \forall i\in[n], \]
the gradient flow initialized at ${\bf D x}_0$ will converge to ${\bf D}\hat{\bf x}$.
\end{lemma}
\begin{proof}


By the results of Lemma \ref{lem:pf5-3}, we obtain
\begin{align}\label{eqn:pf5-6}
    \nabla f({\bf Dx}) = {\bf D} \nabla f({\bf x}),\quad \forall {\bf x} \in\mathbb{R}^n.
\end{align}
%

Hence, we know that the gradient flow initialized with ${\bf Dx}_0$ is equal to ${\bf Dx}(t)$ at time $t$, for all $t\geq 0$. This leads to the conclusion that the new gradient flow will converge to ${\bf D\hat{x}}$. 
\end{proof}
%

\begin{proof}[Proof of Lemma \ref{lem:gd-1}]
Since it is already known that the problem \eqref{eqn:obj-general} has exponentially many global solutions, it remains to prove that the gradient flow with a radial random initialization will converge to one of the above global solutions with equal probability, i.e., with probability $2^{-\lceil n/2 \rceil}$. It has been proved in \cite{lee2016gradient} that the gradient flow will only converge to local minima if the objective function does not have degenerate saddle points, i.e., the Hessian of every saddle point has a negative curvature. Since the global solutions of the problem \eqref{eqn:obj-general} are symmetric with respect to a radial probability distribution, it follows from Lemma 5 that we only need to prove that the objective function of this problem does not have degenerate saddle points. Equivalently, we prove that all second-order critical points are global minima. 

Suppose that $\hat{{\bf x}}$ is a second-order critical point of the problem \eqref{eqn:obj-general} that is not a global minimizer. Let ${\bf M} := {\bf xx}^T$. Due to the symmetry of the landscape, we can assume without loss of generality that
\[ \hat{x}_k \geq 0,\quad k=1,\dots,n. \]
We define the direction $\Delta\in\mathbb{R}^n$ as
\[ \Delta_{2k+1} = \hat{x}_{2k+1} - 1,\quad k=1,\dots,\lceil n/2 \rceil,\quad \Delta_{2k} = \hat{x}_{2k},\quad k=1,\dots, \lfloor n/2 \rfloor. \]
Then, Lemma 7 in \cite{ge2017no} implies that
\begin{align*} 
\Delta : \nabla^2 f\left[ \hat{{\bf x}} \right] : \Delta=&\Delta\Delta^T : \nabla^2 g\left[ \left( \hat{{\bf x}}\hat{{\bf x}}^T - {\bf x}^*{\bf x}^* \right)_\Omega \right] : \Delta\Delta^T - 3({\bf M} - {\bf M}^*) : \nabla^2 g\left[ \left( \hat{{\bf x}}\hat{{\bf x}}^T - {\bf x}^*{\bf x}^* \right)_\Omega \right] : ({\bf M} - {\bf M}^*)\\
\leq& (1+\delta) \|\left(\Delta\Delta^T\right)_\Omega\|_F^2 - 3(1-\delta) \| \left({\bf M} - {\bf M}^*\right)_\Omega\|_F^2,
\end{align*}
where the last inequality is from the assumption that $g(\cdot)$ satisfies the sparse RIP condition with $(\delta,1)$. Combining with the second-order necessary optimality condition, we obtain
\begin{align}\label{eqn:pf5-2}
    (1+\delta) \|\left(\Delta\Delta^T\right)_\Omega\|_F^2 \geq 3(1-\delta) \| \left({\bf M} - {\bf M}^*\right)_\Omega\|_F^2.
\end{align}
Using the expression
\[ \| \left({\bf M} - {\bf M}^*\right)_\Omega\|_F^2 - \|\left(\Delta\Delta^T\right)_\Omega\|_F^2 = \sum_{k=1}^{\lceil n/2\rceil}\left[4\hat{x}_{2k+1}\cdot(\hat{x}_{2k+1}-1)^2 + \|\hat{{\bf x}}^e\|^2 (2\hat{x}_{2k+1} - 1) \right], \]
the inequality \eqref{eqn:pf5-2} can be written as
\begin{align} 
\label{eqn:pf5-3}
-\sum_{k=1}^{\lceil n/2\rceil}\left[4\hat{x}_{2k+1}\cdot(\hat{x}_{2k+1}-1)^2 + \|\hat{{\bf x}}^e\|^2 (2\hat{x}_{2k+1} - 1) \right] \geq  \frac{2-4\delta}{1+\delta} \| \left({\bf M} - {\bf M}^*\right)_\Omega\|_F^2.
\end{align}
The above inequality gives that
\begin{align*}
    \frac{2-4\delta}{1+\delta} \| \left({\bf M} - {\bf M}^*\right)_\Omega\|_F^2 &\leq -\sum_{k=1}^{\lceil n/2\rceil}\left[4\hat{x}_{2k+1}\cdot(\hat{x}_{2k+1}-1)^2 + \|\hat{{\bf x}}^e\|^2 (2\hat{x}_{2k+1} - 1) \right]\\
    &\leq -\sum_{k=1}^{\lceil n/2\rceil}\left[0 - \|\hat{{\bf x}}^e\|^2 \right] = \lceil n/2 \rceil \cdot \|\hat{{\bf x}}^e\|^2 \leq (n+1)/2 \cdot \|\hat{{\bf x}}^e\|^2 \\
    &\leq  (n+1)/2 \cdot \|\hat{{\bf x}}^e\|\|\hat{{\bf x}}\| \leq \sqrt{2}(n+1)\delta \| \left({\bf M} - {\bf M}^*\right)_\Omega\|_F,
\end{align*}
where the second inequality is from the assumption that $\hat{x}_{2k+1} \geq 0$ and the second last inequality is from Lemma \ref{lem:pf5-1}. The above inequality implies that
\[ \| \left({\bf M} - {\bf M}^*\right)_\Omega\|_F \leq \frac{\sqrt{2}(n+1)\delta(1+\delta)}{2-4\delta}. \]
Recalling the condition
\[ n\geq3,\quad \delta \leq \frac{1}{2n}, \]
we obtain
\[ \| \left({\bf M} - {\bf M}^*\right)_\Omega\|_F \leq \frac{1}{2}. \]
Checking the diagonal entries of ${\bf M} - {\bf M}^*=\hat{{\bf x}}\hat{{\bf x}}^T - {\bf x}^*({\bf x}^*)^T$, we have
\[ | \hat{x}_{2k+1} - 1 | \leq \frac12,\quad k=1,\dots,\lceil n/2\rceil, \]
which gives
\[ \hat{x}_{2k+1} \geq \frac12,\quad k=1,\dots,\lceil n/2\rceil. \]
Applying this condition to inequality \eqref{eqn:pf5-3}, the left-hand side of the inequality is non-positive while the right-hand side is non-negative, which implies that
\[ \| \left({\bf M} - {\bf M}^*\right)_\Omega\|_F = 0. \]
This contradicts the assumption that $\hat{{\bf x}}$ is not a global solution. Hence, we have completed the proof that all second-order critical points of the problem \eqref{eqn:obj-general} are global minima. 

Furthermore, using Lemma \ref{lem:pf5-2}, we know that the region of attraction (ROA) of each global minimum is symmetrical. Since the randomly initialized gradient flow converges to a second-order critical point with probability $1$ and all second-order critical points are global minima, the gradient flow with a radial random initialization will converge to each global minimum with equal probability.
\end{proof}

\subsection{Proof of Theorem \ref{thm:gd-1}}

\begin{proof}

In the unperturbed case, the curvature of the Hessian at each global minimum is given by
\[ \Delta : \nabla^2 f({\bf x}) : \Delta \geq (1-\delta)\|\left({\bf x}\Delta^T + \Delta {\bf x}^T\right)_\Omega\|_F^2,\quad \forall \Delta \in \mathbb{R}^{n},~{\bf x}\in\mathcal{X}^*, \]
which has been proved to be positive in Section \ref{sec:exp-sol}. Therefore, we know that the Hessian at each global solution is positive definite and global minima are asymptotically stable for the gradient flow. We choose $R > 0$ to be a large enough constant such that
\[ \mathbb{P}\left[ \|{\bf x}_0\| \leq R \right] \geq 1 - 2^{-\lceil n/2 \rceil}, \]
where the probability is chosen with respect to the initialization distribution. We consider the level set
\[ \mathcal{L}_R := \{ {\bf x}\in\mathbb{R}^{n}~|~ f({\bf x}) \leq c_R \}, \]
where $c_R := \max\{ f({\bf x})~|~ \|{\bf x}\| \leq R \}$. Since the function $f({\bf x})$ is continuous and coercive, the level set $\mathcal{L}_R$ is compact and the gradient flow will not leave $\mathcal{L}_R$ if it is initialized inside it. In addition, it holds that
\[ \mathbb{P}\left[ {\bf x}_0 \in \mathcal{L}_R \right] \geq 1 - 2^{-\lceil n/2 \rceil}. \]
Conditioning on the event that ${\bf x}_0 \in \mathcal{L}_R$, Lemma \ref{lem:gd-1} implies that the gradient flow will converge to each global minimum with the same probability. Let $\hat{{\bf x}}\in\mathcal{X}^*$ be an arbitrary global minimum and $\mathcal{R}_{\hat{{\bf x}}}$ be its ROA of the gradient flow on $\mathcal{L}_{R}$. Therefore, Lemma \ref{lem:gd-1} implies that
\[ \mathbb{P}\left[ {\bf x}_0 \in \mathcal{R}_{\hat{{\bf x}}} ~|~ {\bf x}_0 \in \mathcal{L}_R \right] = 2^{-\lceil n/2\rceil}. \]
By Theorem 4.17 in \cite{Khalil:1173048}, there exist a smooth positive definite function $V({\bf x})$ and a continuous positive definite function $W({\bf x})$ such that every level set of $V({\bf x})$ is compact and
\begin{align*}
    V({\bf x}) \rightarrow +\infty,\quad &\forall {\bf x} \rightarrow \partial\mathcal{R}_{\hat{{\bf x}}},\\
    \left\langle \frac{d V({\bf x})}{d{\bf x}}, -\nabla_{\bf x} f({\bf x}) \right\rangle \leq -W({\bf x}),\quad &\forall {\bf x} \in \mathcal{R}_{\hat{{\bf x}}},
\end{align*}
where ${\bf x} \rightarrow \partial\mathcal{R}_{\hat{{\bf x}}}$ means that the distance between ${\bf x}$ and $\partial\mathcal{R}_{\hat{{\bf x}}}$ goes to zero, and $\partial\mathcal{R}_{\hat{{\bf x}}}$ denotes the boundary of the region of attraction of the solution $\hat{\bf x}$. We choose a large enough constant $M$ such that
\[ \mathbb{P}\left[ {\bf x}_0 \in \mathcal{V}_M  ~|~ {\bf x}_0 \in \mathcal{R}_{\hat{{\bf x}}} \right] \geq 1- 2^{-\lceil n/2\rceil}, \]
where we define the level set $\mathcal{V}_M := \{{\bf x}\in\mathbb{R}^n~|~V({\bf x}) \leq M\}$. Since the level set $\mathcal{V}_M$ is compact, there exists a small enough constant $\epsilon_0$ such that
\[ W({\bf x}) \geq \epsilon_0,\quad \forall {\bf x} \in \mathcal{V}_M. \]

Now, we consider the perturbed case. We denote the new objective function as $\tilde{f}({\bf x}; \eta)$, where $\eta\in\mathbb{R}$ is the perturbation to the global solution. More explicitly, the objective function is defined as
\[ \tilde{f}({\bf x}; \eta) := \left\| \left( {\bf xx}^T - ({\bf x}^*+\eta)({\bf x}^*+\eta)^T \right)_\Omega \right\|_F^2. \]
It has been proved in Section \ref{sec:exp-sol} that the global minimum of the perturbed problem is unique up to a sign flip if the perturbation is sufficiently small and generic, and that there exist $2^{\lceil n/2\rceil} - 2$ spurious local minima. Since the gradient of $\tilde{f}({\bf x}; \eta)$ is a uniformly continuous function of $\eta$ on the compact set $\mathcal{V}_M$, there exists a small enough $r>0$ such that for any generic $\eta_0$ satisfying $\|\eta_0\|\leq r$, it holds that
\[ \left\langle \frac{d V({\bf x})}{d{\bf x}}, -\nabla \tilde{f}({\bf x};\eta_0) \right\rangle \leq -\epsilon_0/2 < 0,\quad \forall {\bf x} \in \mathcal{V}_M.  \]
This implies that the gradient flow on the perturbed problem will not leave $\mathcal{V}_M$ and will converge to a local minimum inside $\mathcal{V}_M$ if ${\bf x}_0 \in \mathcal{V}_M$. Therefore, if the initial point ${\bf x}_0$ is initialized with the given distribution, we have
\[ \mathbb{P}\left[ \lim_{t\rightarrow +\infty} {\bf x}(t) \in \mathcal{V}_M \right] \geq 2^{-\lceil n/2\rceil}\left(1 - 2^{-\lceil n/2\rceil}\right)^2 \geq 2^{-\lceil n/2\rceil}\left(1 - 2^{-\lceil n/2\rceil + 1}\right). \]
By choosing $r$ to be the minimum over all points $\hat{{\bf x}}\in\mathcal{X}^*$, the gradient flow on the perturbed problem will converge to a spurious minimum with probability at least
\[ \left(2^{\lceil n/2\rceil} - 2\right) \cdot 2^{-\lceil n/2\rceil}\left(1 - 2^{-\lceil n/2\rceil + 1}\right) = 1 - \mathcal{O}\left( 2^{-\lceil n/2\rceil} \right). \]
Thus, we can conclude that the gradient flow on the perturbed problem will fail with probability at least $1-\mathcal{O}(2^{-\lceil n/2\rceil})$.

\end{proof}

\end{document}